\documentclass{article}

\usepackage{mathptmx}

\usepackage{geometry}
\usepackage{amssymb}
\usepackage{latexsym, amsmath, amscd,amsthm}
\usepackage{graphicx}
\usepackage[percent]{overpic}
\usepackage{units}
\usepackage{xcolor}
\definecolor{linkblue}{HTML}{003d73}
\definecolor{linkgreen}{HTML}{006161}
\definecolor{linkred}{HTML}{a11950}
\usepackage[pagebackref]{hyperref}
\hypersetup{
	pdftitle={New Stick Number Bounds from Random Sampling of Confined Polygons},
	pdfauthor={Thomas D. Eddy and Clayton Shonkwiler},
	pdfsubject={knot theory},
	pdfkeywords={knots, random polygons, symplectic geometry, stick number},
	colorlinks=true,
	linkcolor=linkblue,
	citecolor=linkgreen,
	urlcolor=linkred
}
\PassOptionsToPackage{caption=false,labelformat=empty}{subfig}
\usepackage[lofdepth]{subfig}
\usepackage[export]{adjustbox}
\usepackage{algorithm}
\usepackage{algorithmicx}
\usepackage{algpseudocode}
\usepackage{booktabs}
\usepackage{authblk}

\usepackage{supertabular,multicol,ifthen}

\makeatletter
\let\mcnewpage=\newpage
\newcommand{\TrickSupertabularIntoMulticols}{%
\renewcommand\newpage{%
    \if@firstcolumn%
        \hrule width\linewidth height0pt%
            \columnbreak%
        \else%
          \mcnewpage%
        \fi%
}%
}
\makeatother

\newtheorem{theorem}{Theorem}

\newtheorem{corollary}[theorem]{Corollary}

\newtheorem*{sticknumber}{Theorem~\ref*{thm:stick numbers}}
\newtheorem*{newbounds}{Theorem~\ref*{thm:new bounds}}

\theoremstyle{definition}

\newcommand{\R}{\mathbb{R}}

\newcommand{\Pol}{\operatorname{Pol}}
\newcommand{\stick}{\operatorname{stick}}
\newcommand{\eqstick}{\operatorname{eqstick}}
\newcommand{\superbridge}{\operatorname{sb}}

\newcommand{\bm}{\boldmath}

\newenvironment{coordinates}[1]{
	\nobreak\vfil\penalty0\vfilneg\vtop\bgroup
	\begin{center} \begin{normalsize} #1 \end{normalsize} \end{center} 

		\ttfamily \begin{tiny}\begin{center}
	}{
		\end{center}\end{tiny}\par
		\xdef\tpd{\the\prevdepth}\egroup\prevdepth=\tpd
	}

\title{New Stick Number Bounds from Random Sampling of Confined Polygons}
\author{Thomas D.\ Eddy}
\author{Clayton Shonkwiler}
\affil{Department of Mathematics, Colorado State University, Fort Collins, CO}
\date{}

\begin{document}

\maketitle

\begin{abstract}
	The stick number of a knot is the minimum number of segments needed to build a polygonal version of the knot. Despite its elementary definition and relevance to physical knots, the stick number is poorly understood: for most knots we only know bounds on the stick number. We adopt a Monte Carlo approach to finding better bounds, producing very large ensembles of random polygons in tight confinement to look for new examples of knots constructed from few segments. We generated a total of 220 billion random polygons, yielding either the exact stick number or an improved upper bound for more than 40\% of the knots with 10 or fewer crossings for which the stick number was not previously known. We summarize the current state of the art in \autoref{tab:stick numbers}, which gives the best known bounds on stick number for all knots up to 10 crossings.
\end{abstract}

\section{Introduction}

The stick number of a knot is the minimum number of segments needed to create a polygonal version of the knot. The equilateral stick number is defined similarly, though with the added restriction that all the segments should be the same length. Despite their elementary definitions, these invariants are only poorly understood: for example, prior to our work the exact stick number was only known for 30 of the 249 nontrivial knots up to 10 crossings.

The sheer simplicity of its definition partially explains the appeal of the stick number as a knot invariant, but the stick number is also interesting as a prototypical \emph{geometric} or \emph{physical} knot invariant since the number of segments needed to build a knot is a measure of the geometric complexity of the knot. Polygonal knots serve as a model of ring polymers like bacterial DNA (see the excellent survey~\cite{Orlandini:2007kn}) so, while this model is not particularly realistic without adding further constraints, the stick number gives an indication of the minimum size of a ring polymer (or, indeed, any knotted physical system) needed to achieve a given knot type. 

Any example of a polygonal knot gives an upper bound on the stick number of the knot: if we can build a polygonal version of a knot $K$ out of 10 sticks, then the stick number of $K$ cannot possibly be bigger than 10. Consequently, we can prove state-of-the-art bounds on stick numbers by constructing examples of knots with fewer sticks than previously observed. If we happen to construct an example of a knot $K$ with $n$ sticks where $n$ is known to be a lower bound on the stick number of $K$, then this proves that the stick number of $K$ is exactly $n$.

Our approach to finding such examples is to randomly sample the space of equilateral $n$-stick knots for small $n$, which requires identifying the knot type of each sample and checking to see if it beats the previous record. Early work of Millett~\cite{Millett:1994fo,Millett:2000fe} and Calvo and Millett~\cite{Calvo:1998kr} used a similar Monte Carlo approach, but the challenge is that the overwhelming majority of samples are not interesting, so getting good results requires generating a truly vast number of samples. Indeed, the best extant bounds on (equilateral) stick number are due to Scharein~\cite{Scharein:1998tu} and Rawdon and Scharein~\cite{Rawdon:2002wj}, who used a combination of stochastic agitation and relaxation starting from standard (rather than random) conformations.

\begin{figure}[t]
	\centering
		\subfloat[$9_{35}$]{\includegraphics[scale=.25,valign=c]{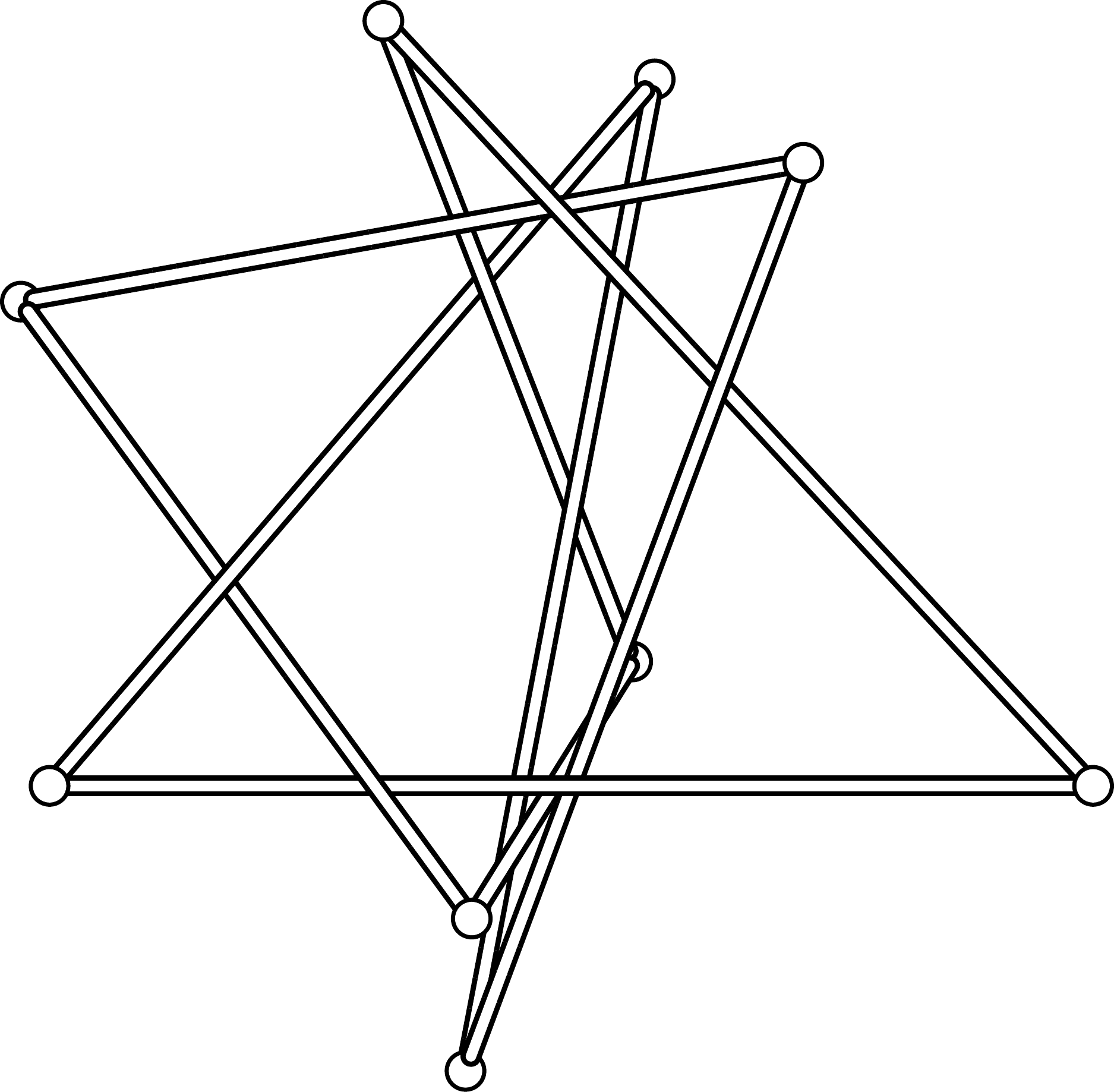}
		 \vphantom{\includegraphics[scale=.25,valign=c]{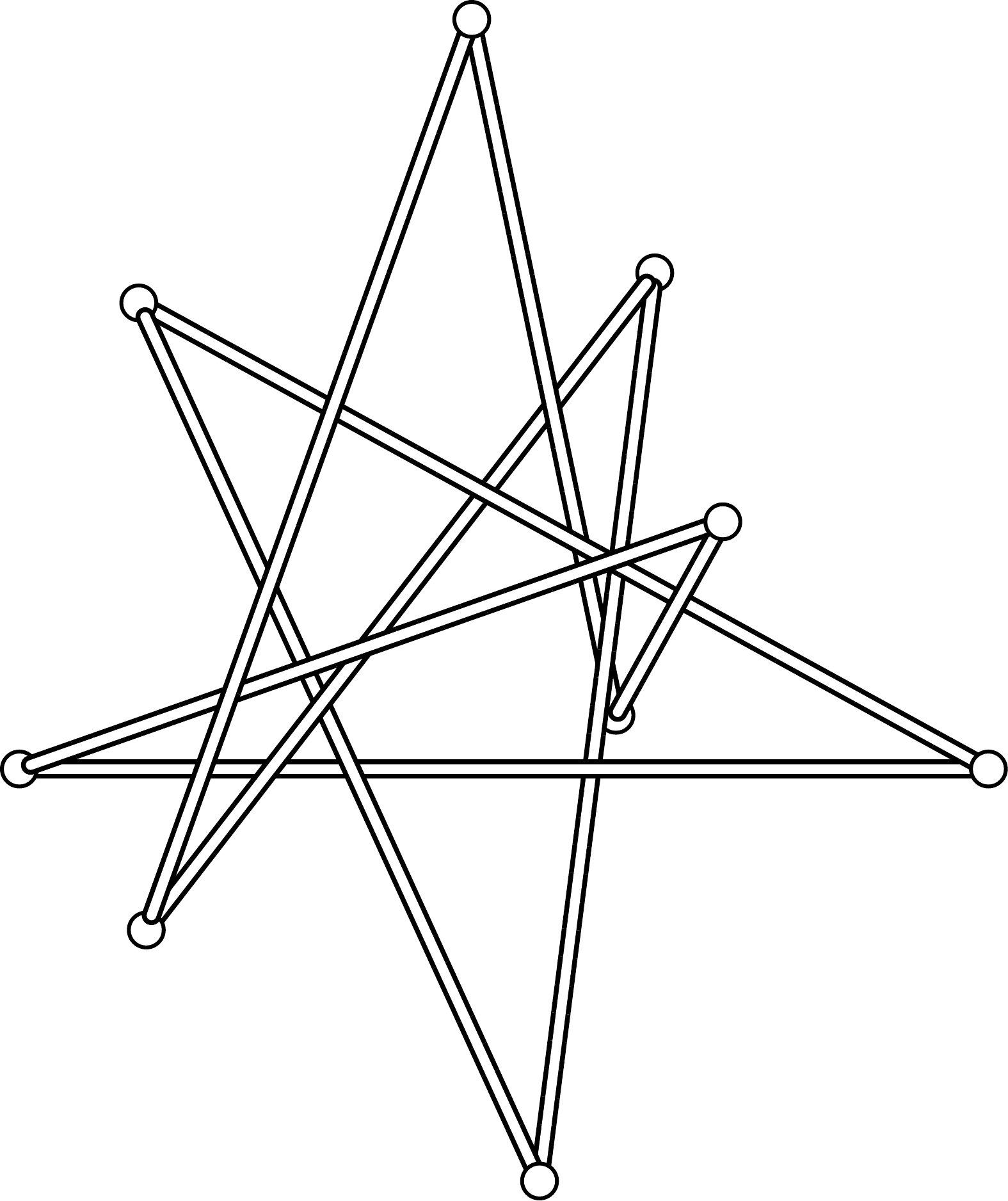}}}\qquad
		\subfloat[$9_{39}$]{\includegraphics[scale=.25,valign=c]{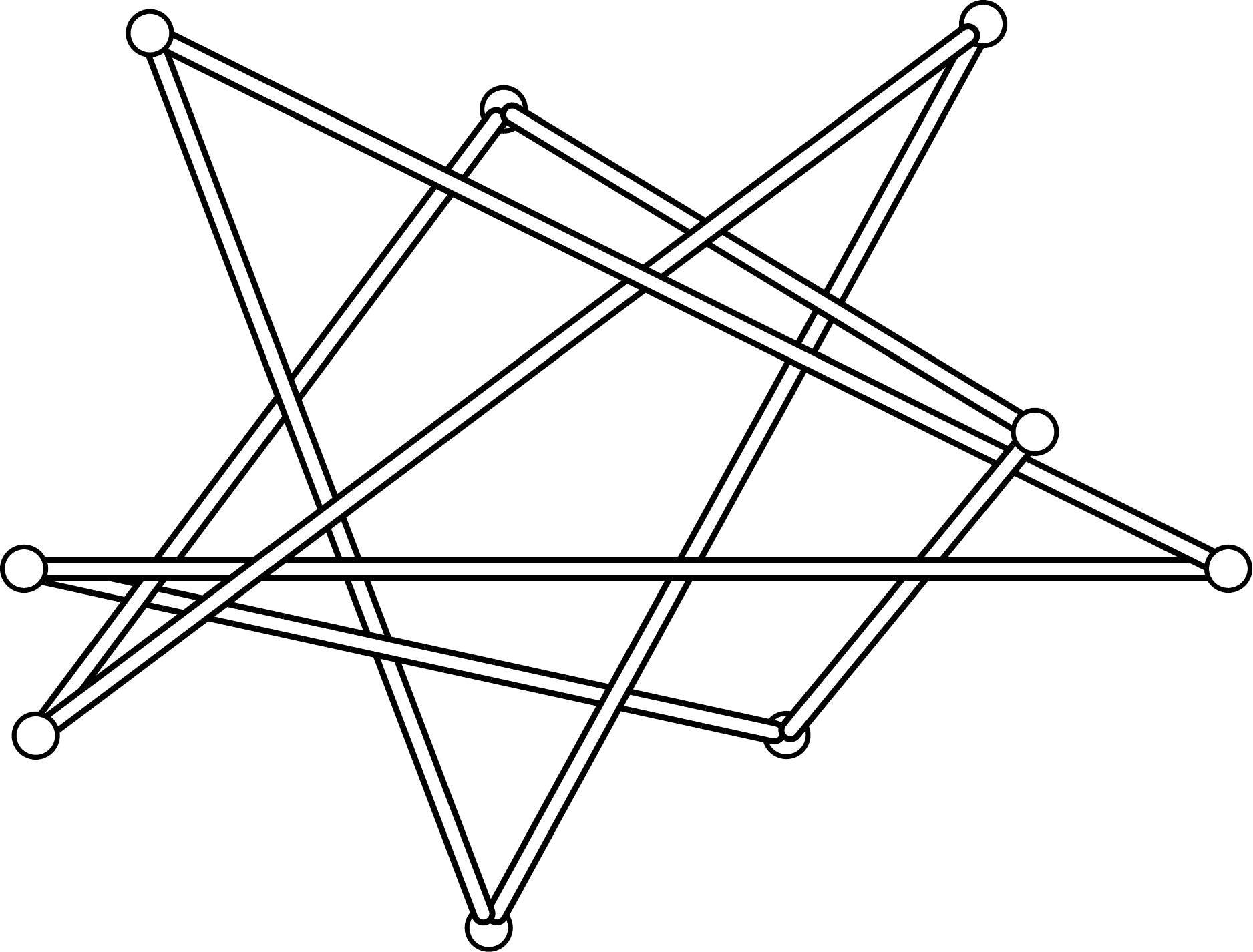}
		 \vphantom{\includegraphics[scale=.25,valign=c]{9_43z.pdf}}}\qquad
		\subfloat[$9_{43}$]{\includegraphics[scale=.25,valign=c]{9_43z.pdf}} 
		
		\subfloat[$9_{45}$]{\includegraphics[scale=.25,valign=c]{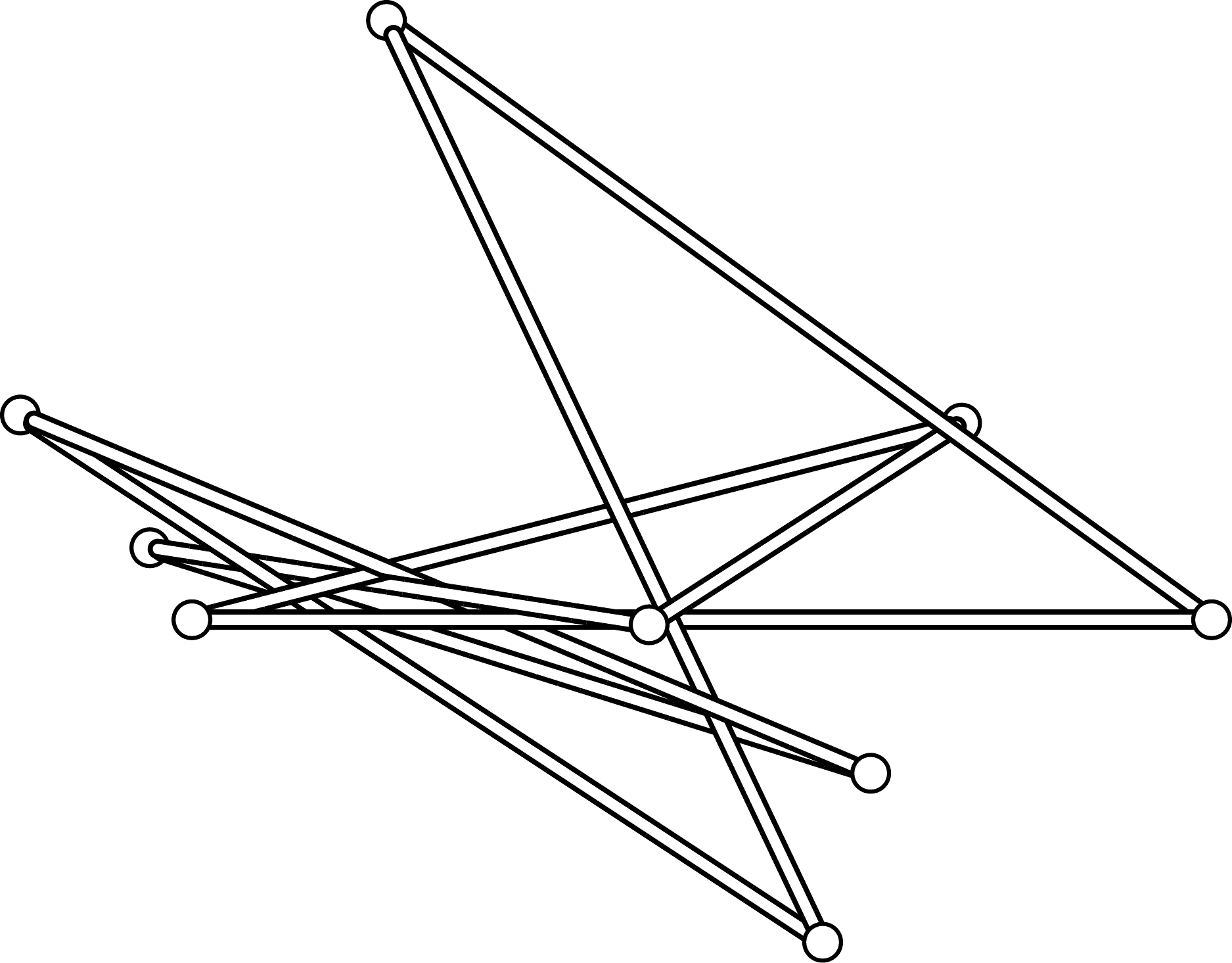}
		 \vphantom{\includegraphics[scale=.25,valign=c]{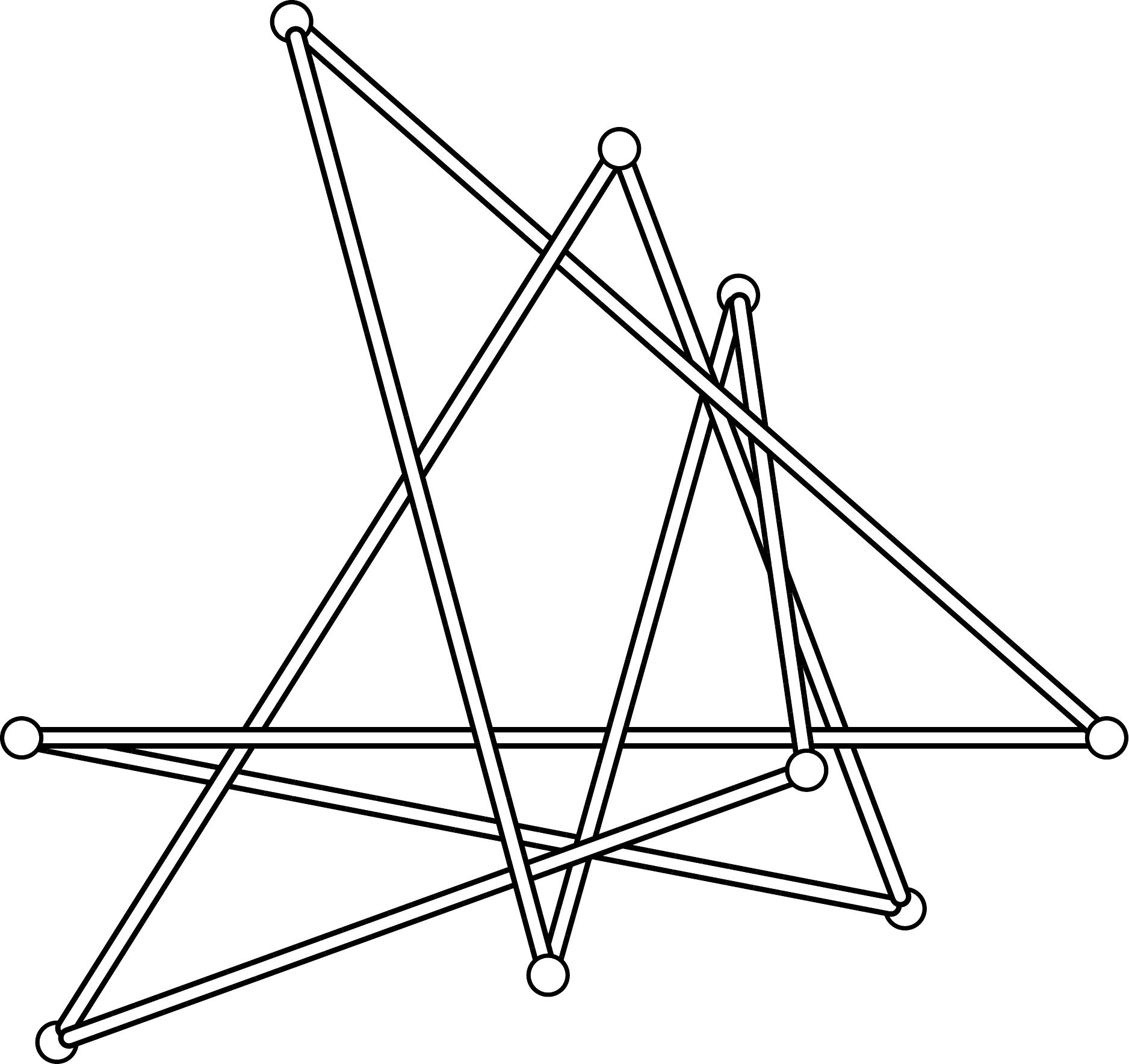}}}\qquad
		\subfloat[$9_{48}$]{\includegraphics[scale=.25,valign=c]{9_48z.pdf}}
	\caption{Minimal stick equilateral representations of the knots $9_{35}$, $9_{39}$, $9_{43}$, $9_{45}$, and $9_{48}$. Each knot is shown in orthographic perspective, viewed from the direction of the positive $z$-axis.}
	\label{fig:9sticks}
\end{figure}

The main novelty of our Monte Carlo approach is that we generate polygonal knots with equal-length edges in very tight confinement. Since complicated knots are typically condensed, this approach is a form of enriched sampling which boosts the yield of the complicated knots we seek. Generating confined knots to explore stick number is not new in itself -- Calvo and Millett~\cite{Calvo:1998kr,Millett:2000fe} generated stick knots with vertices uniformly distributed in the unit cube (the so-called \emph{uniform random polygon} model) -- but ours seems to be the first such investigation which addresses \emph{equilateral} stick number. Our strategy depends, in turn, on recent advances in sampling random equilateral polygons using symplectic geometry~\cite{Cantarella:2016iy,Cantarella:2016bt}. 

The efficiency of our algorithms combined with modern hardware leads to a substantial improvement over earlier Monte Carlo experiments in the sheer size of the ensembles we produce. We generated a total of 220 billion random knots, requiring approximately 50,000 core-hours of processing time on a 64-core machine with 2.3 GHz AMD Opteron 6276 processors. The bulk of this time was devoted to identifying the knot type of each of these 220 billion knots, which we did using the HOMFLY polynomial, hyperbolic volume, and crossing number.

We found examples of the knots\footnote{For knots up to 10 crossings we use Alexander--Briggs notation consistent with the (Perko-corrected) Rolfsen table~\cite{Rolfsen:2003ts}; for knots with 11 or more crossings, we use the Dowker--Thistlethwaite naming convention.} $9_{35}$, $9_{39}$, $9_{43}$, $9_{45}$, and $9_{48}$ built with 9 equal-length sticks, shown in \autoref{fig:9sticks}; since the stick number of these knots is known to be at least 9, this gives the exact (equilateral) stick number for these five knots:

\begin{theorem}\label{thm:stick numbers}
	The stick number and equilateral stick number of each of the knots $9_{35}$, $9_{39}$, $9_{43}$, $9_{45}$, and $9_{48}$ is exactly 9.
\end{theorem}

As we will see, this implies that each of these knots has superbridge index 4.

Also, we found new upper bounds on the equilateral stick number for 88 of the 214 knots up to 10 crossings for which it remains unknown:

\begin{theorem}\label{thm:new bounds}The equilateral stick number of each of the knots $9_2$, $9_3$, $9_{11}$, $9_{15}$, $9_{21}$, $9_{25}$, $9_{27}$, $10_8$, $10_{16}$, $10_{17}$, $10_{56}$, $10_{83}$, $10_{85}$, $10_{90}$, $10_{91}$, $10_{94}$, $10_{103}$, $10_{105}$, $10_{106}$, $10_{107}$, $10_{110}$, $10_{111}$, $10_{112}$, $10_{115}$, $10_{117}$, $10_{118}$, $10_{119}$, $10_{126}$, $10_{131}$, $10_{133}$, $10_{137}$, $10_{138}$, $10_{142}$, $10_{143}$, $10_{147}$, $10_{148}$, $10_{149}$, $10_{153}$, and $10_{164}$ is less than or equal to 10.
	
	The equilateral stick number of each of the knots $10_3$, $10_6$, $10_7$, $10_{10}$, $10_{15}$, $10_{18}$, $10_{20}$, $10_{21}$, $10_{22}$, $10_{23}$, $10_{24}$, $10_{26}$, $10_{28}$, $10_{30}$, $10_{31}$, $10_{34}$, $10_{35}$, $10_{38}$, $10_{39}$, $10_{43}$, $10_{44}$, $10_{46}$, $10_{47}$, $10_{50}$, $10_{51}$, $10_{53}$, $10_{54}$, $10_{55}$, $10_{57}$, $10_{62}$, $10_{64}$, $10_{65}$, $10_{68}$, $10_{70}$, $10_{71}$, $10_{72}$, $10_{73}$, $10_{74}$, $10_{75}$, $10_{77}$, $10_{78}$, $10_{82}$, $10_{84}$, $10_{95}$, $10_{97}$, $10_{100}$, and $10_{101}$ is less than or equal to 11.
	
	The equilateral stick number of each of the knots $10_{76}$ and $10_{80}$ is less than or equal to 12.
	
	In particular, all knots up to 10 crossings have equilateral stick number $\leq 12$.
\end{theorem}

An upper bound on equilateral stick number implies an upper bound on stick number, and all of the corresponding stick number bounds in \autoref{thm:new bounds} are new except for $10_{107}$, $10_{119}$, and $10_{147}$. These knots, together with $8_{19}$, $9_{29}$, $10_{16}$, and $10_{79}$, appeared on Rawdon and Scharein's list of knots for which they could not find an equilateral stick knot achieving a known bound on stick number, and thus were potential examples of knots for which stick number and equilateral stick number differ. Millett~\cite{Millett:2012dd} found an example of an equilateral 8-stick $8_{19}$, proving that $\eqstick(8_{19})=8$, which was already known to be the stick number, and our 10-stick $10_{16}$ (see \autoref{fig:big improvements}) beats the previous best bounds on both stick number (11) and equilateral stick number~(12). In turn, our observed equilateral 10-stick examples of $10_{107}$, $10_{119}$, and $10_{147}$ match the best previous bound on stick number, leaving only $9_{29}$ and $10_{79}$ from Rawdon and Scharein's list. 

\begin{figure}[t]
	\centering
		\subfloat[$10_{16}$]{\includegraphics[scale=.3,valign=c]{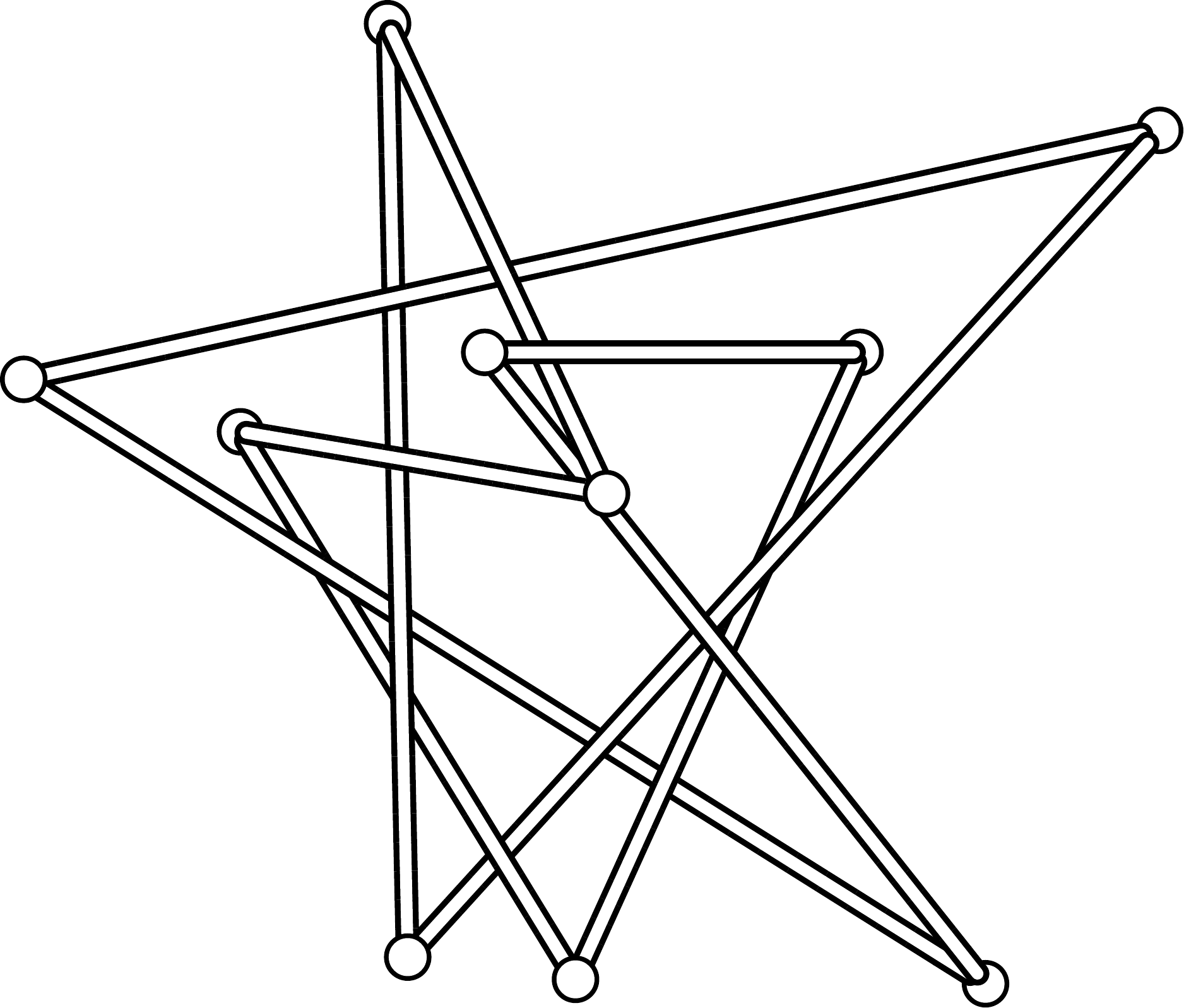}}\qquad\quad 
		\subfloat[$10_{84}$]{\includegraphics[scale=.3,valign=c]{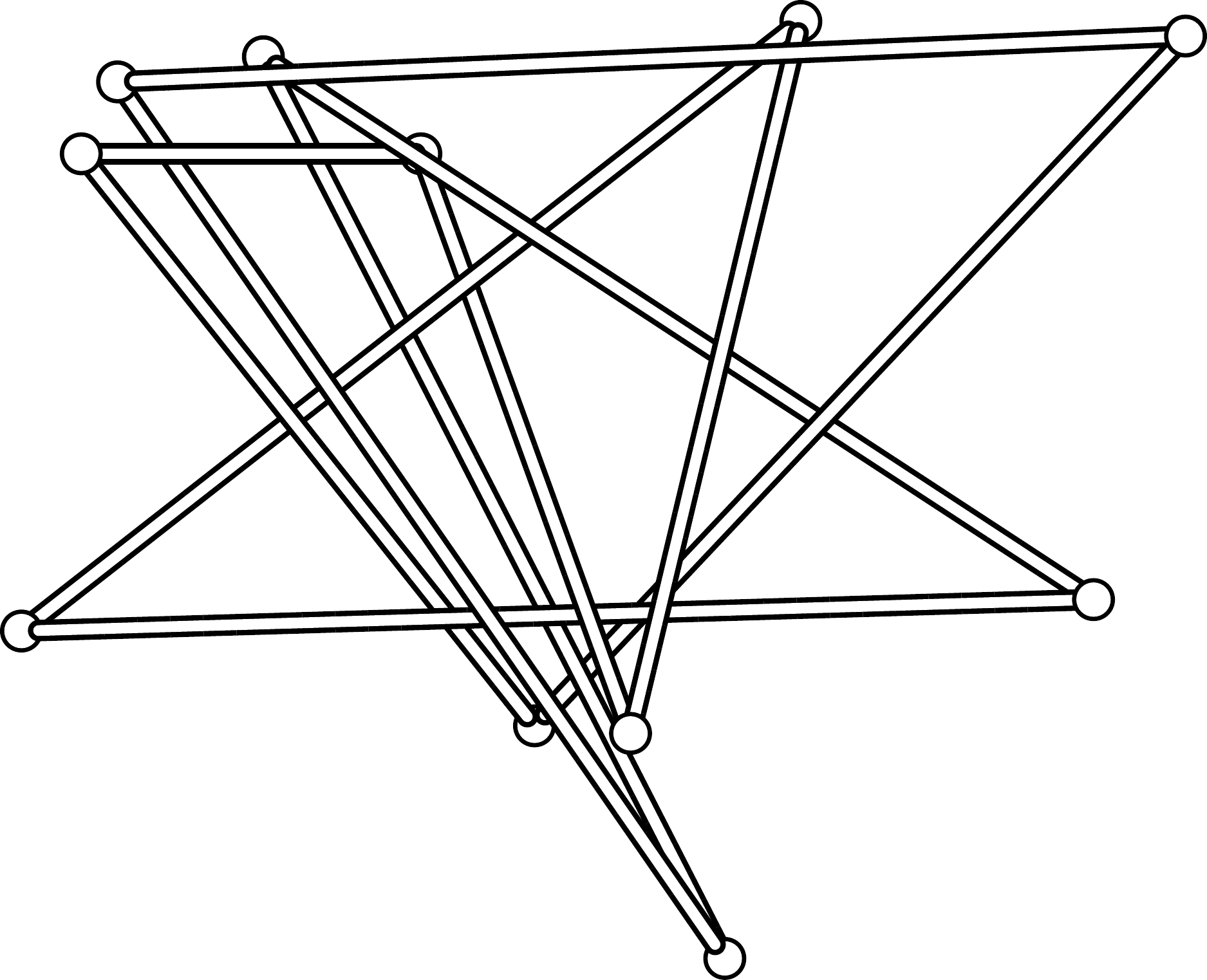}
		\vphantom{\includegraphics[scale=.3,valign=c]{10_16offaxis.pdf}}}
	\caption{The 10-stick $10_{16}$ improves the best previous bound on equilateral stick number by 2. The 11-stick $10_{84}$ improves the best previous bounds on both stick number and equilateral stick number by 3. Each knot is shown in orthographic perspective from the direction $(3,1,1)$.}
	\label{fig:big improvements}
\end{figure}

The new bounds on equilateral stick number for the knots $10_8$, $10_{16}$, $10_{39}$, $10_{64}$, $10_{73}$, $10_{105}$, $10_{110}$, and $10_{117}$ in \autoref{thm:new bounds} are 2 smaller than the previous best bounds due to Rawdon and Scharein, and for $10_{84}$ (see \autoref{fig:big improvements}) the new bound is 3 smaller. These substantial improvements in bounds for certain knots, coupled with the fact that Rawdon and Scharein's bounds are superior to ours for the knots $9_{16}$, $10_{1}$, $10_{109}$, $10_{113}$, $10_{114}$, and $10_{123}$, gives a sense of the differences between their approach and ours, which we view as complementary.

\subsection*{Organization}

In building up to Theorems~\ref{thm:stick numbers} and~\ref{thm:new bounds}, we start with background on stick number and equilateral stick number in \autoref{sec:background} and a review of new approaches to sampling random polygons using symplectic geometry in \autoref{sec:symplectic}. The latter section may be of independent interest, since it gives a concise and reasonably self-contained exposition of the symplectic approach to sampling confined polygons introduced in~\cite{Cantarella:2016iy}. We describe our computational pipeline in \autoref{sec:pipeline} and give a more detailed statement of our results in \autoref{sec:results}, which is supplemented by the tables in the appendices. Finally, we pose some questions and mention some directions for future inquiry in \autoref{sec:conclusion}.

This is work that grew out of the first author's master's thesis~\cite{Eddy:2019uj}, which was based on smaller ensembles of random knots.

\section{Stick Number}\label{sec:background}

A tame, piecewise-linear embedding $S^1 \hookrightarrow \R^3$ is called a \emph{polygonal knot} or a \emph{stick knot}. Any stick knot consists of a finite number of straight segments. Define the \emph{stick number} of a knot $K$, denoted $\stick(K)$, to be the minimum number of linear segments necessary to construct a polygonal version of $K$. The \emph{equilateral stick number}, denoted $\eqstick(K)$, is the minimum number of congruent segments necessary to construct a polygonal version of $K$, and of course $\eqstick(K) \geq \stick(K)$ for any knot type $K$. By construction the stick number and equilateral stick number are knot invariants, though not ones that are generally easy to compute. The stick number gives some measure of the inherent \emph{geometric} (as opposed to purely topological) complexity of a knot, and displays similar behavior as other geometric knot invariants like rope length. For example, the $8_{19}$ knot has shorter rope length than any 7-crossing knot~\cite{Ashton:2011du}, and it also has smaller stick number than any 7-crossing knot (see \autoref{tab:stick numbers}).

The stick number invariant was defined by Randell~\cite{Randell:1994bx}, who showed that $\stick(3_1) = 6$ and $\stick(4_1) = 7$. Randell also showed that every other nontrivial knot must have stick number at least 8, which was an early example of a (in this case lower) bound on stick number. The best general bounds on stick number come from combining work of Negami~\cite{Negami:1991gb}, Huh and Oh~\cite{Huh:2011co}, and Calvo~\cite{Calvo:2001gv}:
\begin{theorem}[Negami~\cite{Negami:1991gb}, Huh and Oh~\cite{Huh:2011co}, and Calvo~\cite{Calvo:2001gv}]\label{thm:stick inequalities}
	For a nontrivial knot $K$ with crossing number $c(K)$,
\[
	\frac{7+\sqrt{8 c(K)+1}}{2} \leq \stick(K) \leq \frac{3}{2}(c(K)+1).
\]
The upper bound can be improved to $\frac{3}{2}c(K)$ for non-alternating prime knots. 
\end{theorem}
The lower bound also applies to the equilateral stick number $\eqstick(K) \geq \stick(K)$, but the best known theoretical upper bound, due to Kim, No, and Oh~\cite{Kim:2014du}, is somewhat worse: $\eqstick(K) \leq 2c(K)+2$ in general, and $\eqstick(K)\leq 2c(K)-2$ for non-alternating prime knots.

These bounds are typically quite loose: both inequalities in \autoref{thm:stick inequalities} are tight for the trefoil and figure eight knots, and the lower bound is tight for the $8_{19}$ and $8_{20}$ knots, but otherwise none of these bounds on $\stick(K)$ and $\eqstick(K)$ are saturated for any prime knot with 10 or fewer crossings, and it would be somewhat surprising if they were ever saturated by any more complicated knots. Tighter bounds can be proved for more restricted classes of knots. For example:

\begin{theorem}[Huh, No, and Oh~\cite{Huh:2011gp}]\label{thm:2-bridge bound}
	For any $2$-bridge knot $K$ with $c(K) \geq 6$,
	\[
		\stick(K) \leq c(K) + 2.
	\]
\end{theorem}

This bound is sharp for all 6- and 7-crossing knots, and is the best upper bound known for the knots $8_{1-4}$, $8_{6-9}$, $8_{11-14}$, $9_6$, $9_{18}$, $9_{23}$, and $10_{37}$, as well as for essentially all 2-bridge knots with more than 10 crossings except $11a_{191}$, which has 11-stick realizations, and $11a_{77, 89, 90, 110, 111, 117, 140, 154, 166, 174, 177, 178, 180, 190, 203, 224, 247}$ and $12a_{537}$, for which we have observed 12-stick examples.

Even better than an upper bound, the exact stick numbers are known for an infinite collection of torus knots:

\begin{theorem}[Adams et al.~\cite{Adams:1997gb}, Jin~\cite{Jin:1997da}, Bennett~\cite{Bennett:2008th}, Adams and Shayler~\cite{Adams:2009kp}, and Johnson et al.~\cite{Johnson:2013ko}]\label{thm:torus bound}
	For relatively prime $2 \leq p < q < 3p$, the stick number of the $(q,p)$-torus knot is $2q$ if $q<2p$ and $4p$ if $2p<q<3p$.
\end{theorem}
In particular, it follows from \autoref{thm:torus bound} that $\stick(8_{19})=8$, $\stick(10_{124})=\stick(15n_{41,185})=10$, and ${\stick(14n_{21,881})=\stick(16n_{783,154})=12}$.

In a different direction, Calvo~\cite{Calvo:1998uj,Calvo:2002iq} did a deep analysis of the space of 8-stick knots, showing:

\begin{theorem}[Calvo~\cite{Calvo:1998uj,Calvo:2002iq}]\label{thm:8sticks}
	The only knots with $\stick(K) \leq 8$ are $3_1$, $4_1$, $5_1$, $5_2$, $6_1$, $6_2$, $6_3$, $8_{19}$, $8_{20}$, $3_1 \# 3_1$, and $3_1 \# -\!3_1$.
\end{theorem}

Said another way, $\stick(K)\geq 9$ for any knot $K$ not in the above list, and if we encounter such a knot which can be realized as an equilateral 9-gon, then we can immediately conclude that ${\stick(K) = \eqstick(K)=9}$.

In practice, most useful upper bounds on the stick numbers of knots come from examples: after all, if you can find a representative of a knot type $K$ with $n$ sticks, then certainly $\stick(K) \leq n$. Meissen~\cite{Meissen:1998wu} found 9-stick representatives of all 7-crossing knots which, combined with Calvo's theorem, implies these knots all have stick number 9, and Millett~\cite{Millett:1994fo,Millett:2000fe} and Calvo and Millett~\cite{Calvo:1998kr} sampled millions of random polygons to find representatives of most knots up to 9 crossings with relatively low stick numbers.

Rob Scharein obtained much better results in his thesis~\cite{Scharein:1998tu} using a combination of stochastic displacements of vertex positions and opportunistic deletions. Rawdon and Scharein~\cite{Rawdon:2002wj} extended this approach to the equilateral stick number, and the Scharein and Rawdon--Scharein bounds have remained the gold standard for almost two decades, especially since they are easily referenced online at \url{http://www.colab.sfu.ca/KnotPlot/sticknumbers/} and examples achieving the bounds are built into Scharein's software {\tt KnotPlot}~\cite{knotplot}.

Despite his protestations that he had ``not been greatly concerned with efficiency''~\cite[p.~146]{Scharein:1998tu}, Scharein's approach gave better results than Calvo and Millett's Monte Carlo explorations because it was more efficient: the subset of all 9-stick knots which form, say, a $9_{35}$ knot is so tiny that one needs a truly huge number of random samples to encounter a point in it, whereas Scharein was starting from a conformation known to be a $9_{35}$ and optimizing it through knot-type-preserving modifications.

However, recent advances in (equilateral) polygon sampling algorithms make it possible to uniformly sample random equilateral stick knots in very tight confinement, and to do so quickly enough that generating hundreds of billions of samples is now feasible. Hence, this is a good time to revisit the Monte Carlo approach to finding upper bounds on stick numbers.

\section{Symplectic Geometry and Sampling Confined Polygons}\label{sec:symplectic}


Stick knots -- which is to say, piecewise-linear embeddings of the circle in 3-space -- are of interest in polymer physics and molecular biology, where they appear as models of ring polymers in solution or in melt. The survey~\cite{Orlandini:2007kn} gives a nice overview of applications of these models in physics and biology. In this context a piecewise-linear embedding of a circle is usually called a \emph{polygon} and there has been a lot of work done on developing efficient algorithms for sampling random polygons. Since a polygon is just a stick knot, we can use such an algorithm to generate large ensembles of random stick knots and look for examples with fewer sticks than previously observed.

The algorithm we will use is based on understanding the geometry of the space of all possible conformations of equilateral polygons up to rigid motions. To get nice coordinates on this space, we will interpret an $n$-gon as an ordered collection of edge vectors $\vec{e}_1, \dots , \vec{e}_n \in \R^3$ rather than as a list of vertices. In particular, this gives a translation-invariant representation of a polygon. 

A list of vectors forms a polygon (as opposed to an open chain) if and only if $\vec{e}_1 + \dots + \vec{e}_n = \vec{0}$, and the polygon is equilateral if $\|\vec{e}_i\|=1$ for all $i=1,\dots,n$. Hence, using $S^2$ to denote the unit sphere in $\R^3$, we will identify the collection of equilateral $n$-gons up to translation as the subset of $\underbrace{S^2 \times \cdots \times S^2}_n$ consisting of those $n$-tuples of unit vectors which sum to zero.

This condition is clearly preserved by the diagonal action of the rotation group $SO(3)$, and passing to the quotient produces the space $\Pol(n)$ of equilateral $n$-gons in $\R^3$ up to rigid motions. When $n$ is even this $SO(3)$ action is not quite free -- degenerate polygons which lie on a line are fixed by any rotation around the line -- so $\Pol(n)$ is not always a manifold, but it has at worst finitely many isolated singular points. Moreover, this construction turns out to be nothing more than a symplectic reduction of $S^2 \times \cdots \times S^2$ by the diagonal $SO(3)$ action~\cite{Kapovich:1996wo}, and hence $\Pol(n)$ has a natural symplectic structure induced by the standard symplectic structure on $S^2 \times \cdots \times S^2$.

In fact, $\Pol(n)$ is more than merely symplectic or even K\"ahler: an open, dense subset $\Pol(n)^\circ$ of $\Pol(n)$ is toric~\cite{Kapovich:1996wo}. A simple dimension count shows that $\dim \Pol(n) = 2n-6=2(n-3)$ since $\vec{e}_1 + \dots + \vec{e}_n = \vec{0}$ is a codimension-3 condition and $SO(3)$ is 3-dimensional, so being toric means that there is a Hamiltonian $(n-3)$-torus action on $\Pol(n)^\circ$ with corresponding moment map $\mu: \Pol(n)^\circ \to \R^{n-3}$. Consequently, so-called \emph{action-angle coordinates} give a global coordinate system on $\Pol(n)^\circ$ which is well-adapted to random sampling.

Specifically, a choice of triangulation of an abstract $n$-gon -- which is to say, a choice of $n-3$ non-intersecting segments connecting non-adjacent vertices of an abstract, convex, planar $n$-gon -- induces an action of the $(n-3)$-torus. In what follows, we focus on the \emph{fan triangulation} shown in \autoref{fig:fan triangulation}, which connects all non-adjacent vertices to a chosen root vertex. 

\begin{figure}[t]
	\centering
		\begingroup
		\setlength{\unitlength}{2.75in}
	    \begin{picture}(1,0.75)
	      \put(0,0){\includegraphics[width=\unitlength]{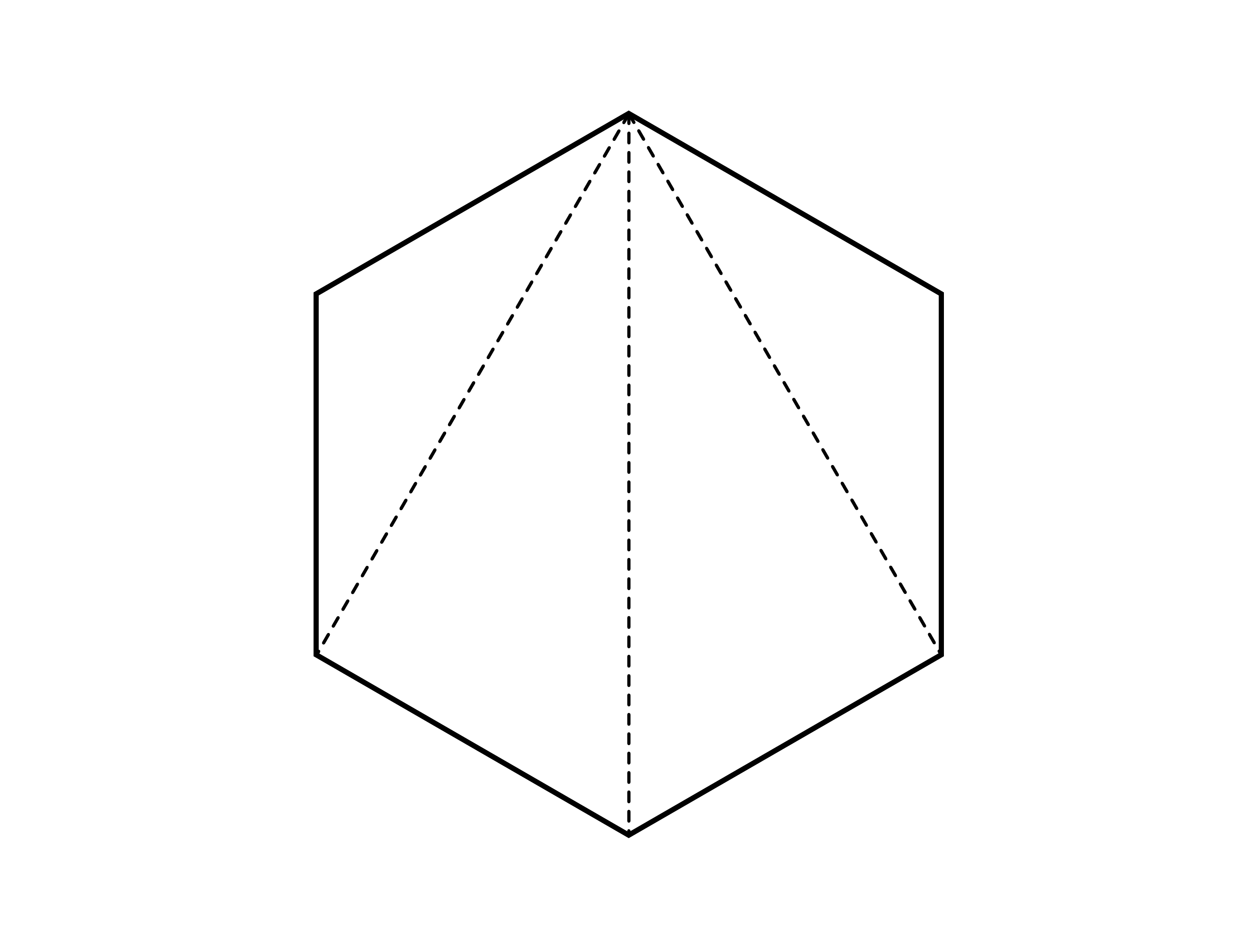}}
	    \end{picture}
	    \begin{picture}(1,0.75)
	      \put(0,0){\includegraphics[width=\unitlength]{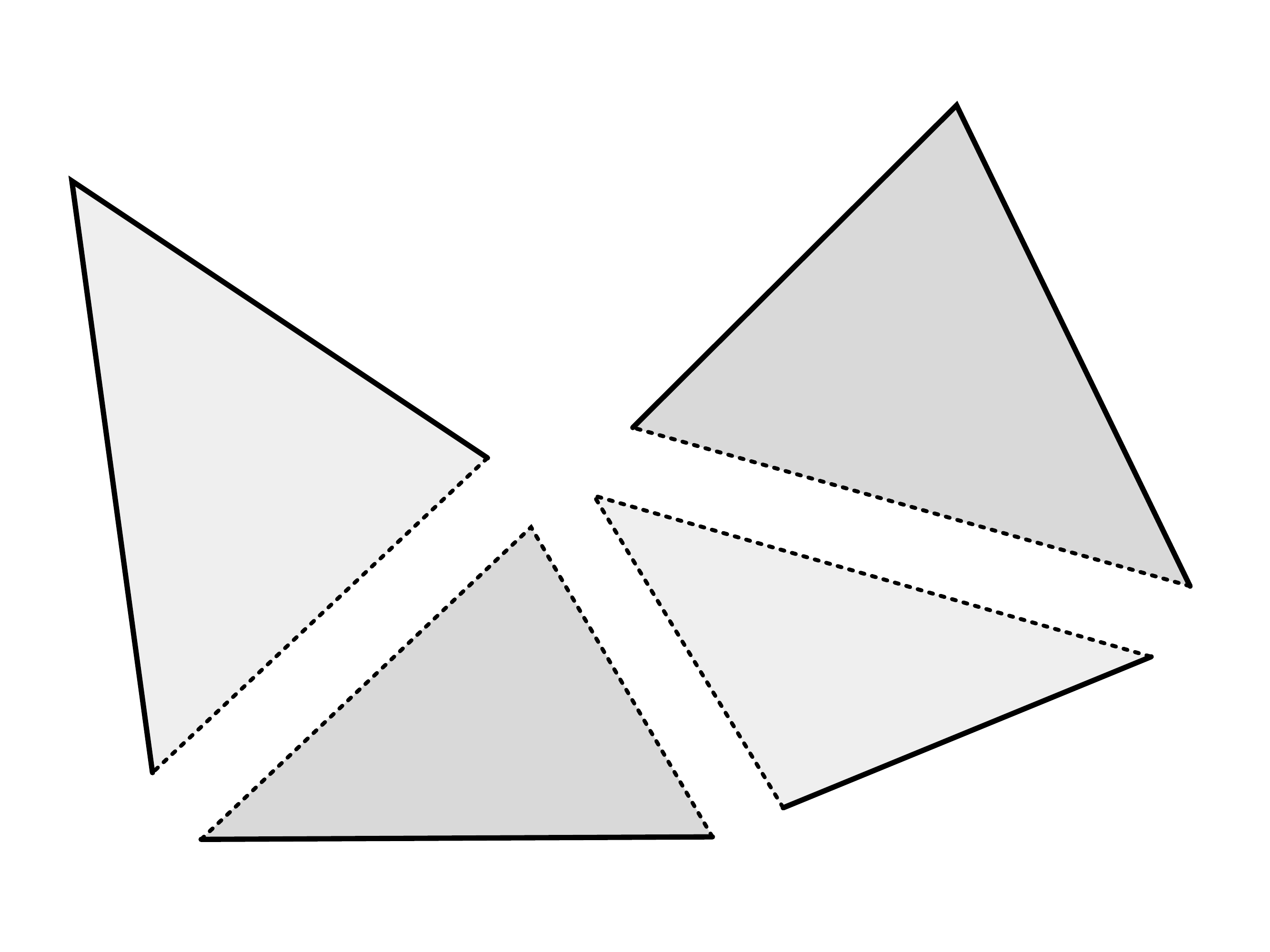}}
	      \put(0.41,0.38){\smash{$v_1$}}
	      \put(0.74,0.69){\smash{$v_2$}}
	      \put(0.92,0.245){\smash{$v_3$}}
	      \put(0.58,0.08){\smash{$v_4$}}
	      \put(0.1,0.1){\smash{$v_5$}}
	      \put(0.02,0.63){\smash{$v_6$}}
	      \put(0.67,0.31){\smash{$d_1$}}
	      \put(0.49,0.22){\smash{$d_2$}}
	      \put(0.24,0.225){\smash{$d_3$}}
	    \end{picture}
	    \begin{picture}(1,0.75)
	      \put(0,0){\includegraphics[width=\unitlength]{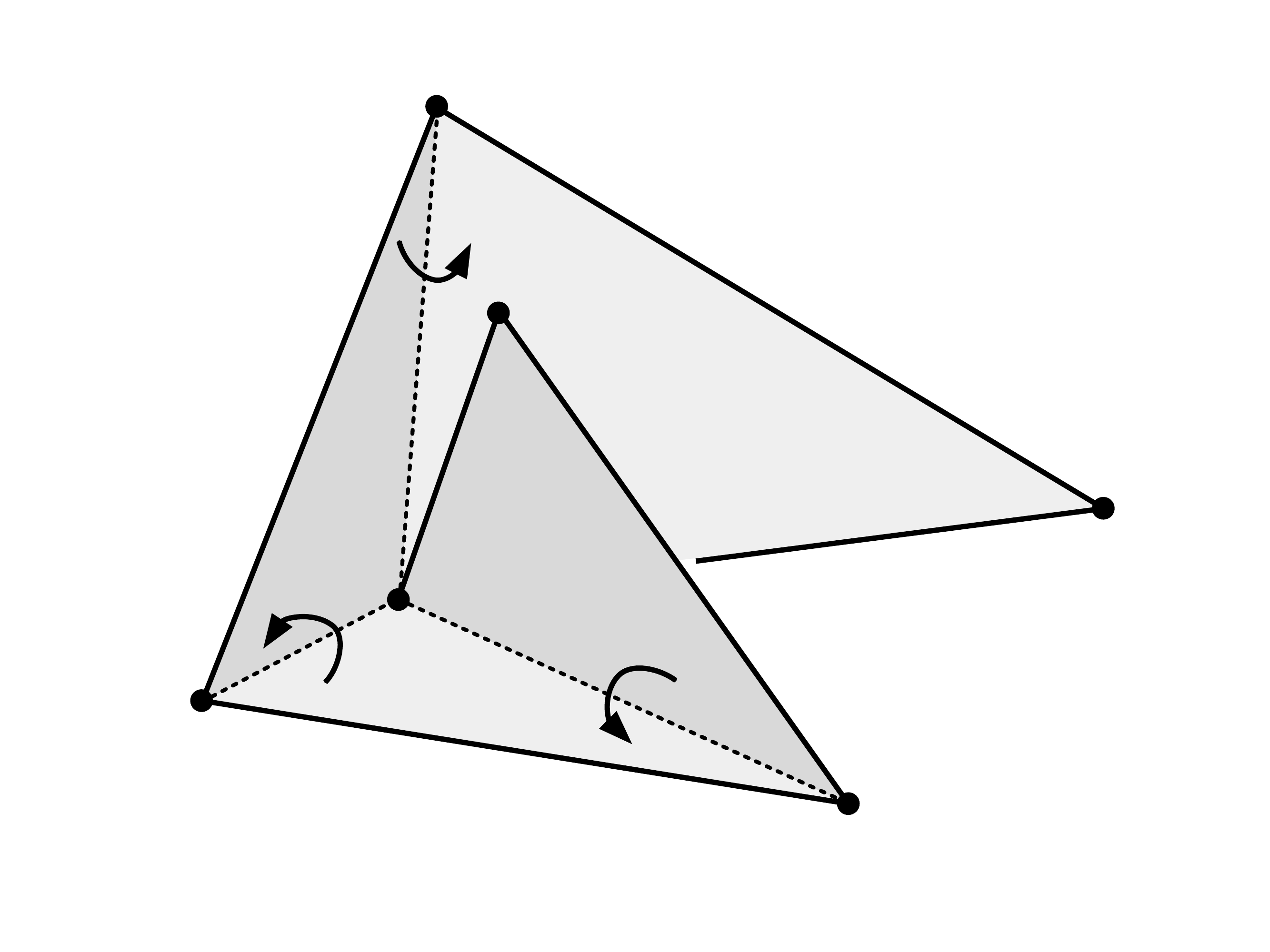}}
	      \put(0.3,0.24){\smash{$v_1$}}
	      \put(0.4,0.52){\smash{$v_2$}}
	      \put(0.67,0.08){\smash{$v_3$}}
	      \put(0.12,0.16){\smash{$v_4$}}
	      \put(0.33,0.685){\smash{$v_5$}}
	      \put(0.88,0.32){\smash{$v_6$}}
	      \put(0.45,0.24){\smash{$\theta_1$}}
	      \put(0.23,0.28){\smash{$\theta_2$}}
	      \put(0.308,0.495){\smash{$\theta_3$}}
	    \end{picture}
	    \begin{picture}(1,0.75)
	      \put(0,0){\includegraphics[width=\unitlength]{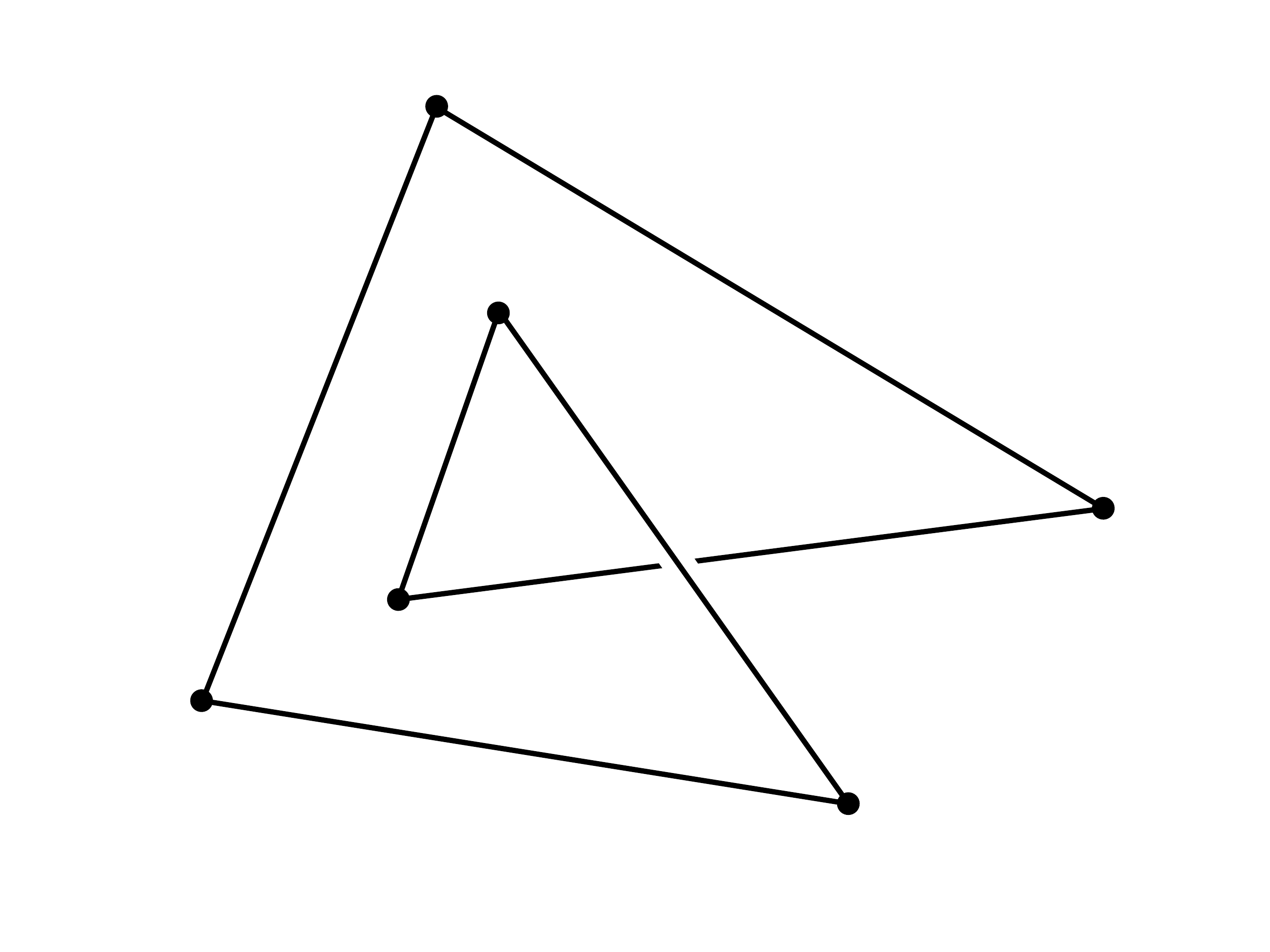}}
	      \put(0.28,0.24){\smash{$v_1$}}
	      \put(0.38,0.52){\smash{$v_2$}}
	      \put(0.67,0.08){\smash{$v_3$}}
	      \put(0.12,0.16){\smash{$v_4$}}
	      \put(0.33,0.685){\smash{$v_5$}}
	      \put(0.88,0.32){\smash{$v_6$}}
	    \end{picture}
		\endgroup
	\caption{The fan triangulation and action-angle coordinates. The $n-3$ chords determining the fan triangulation are shown top left. The $d_i$ determine the dashed lengths in the $n-2$ triangles coming from the triangulation; this, together with the solid sides, which are all unit length, uniquely determine the triangles, and the triangulation tells how to glue them together. The $\theta_i$ determine the dihedral angles of the gluings and removing the interiors of the triangles along with the dashed edges produces an $n$-gon. Starting instead from the bottom right shows how to extract the action-angle coordinates from a particular $n$-gon.}
	\label{fig:fan triangulation}
\end{figure}
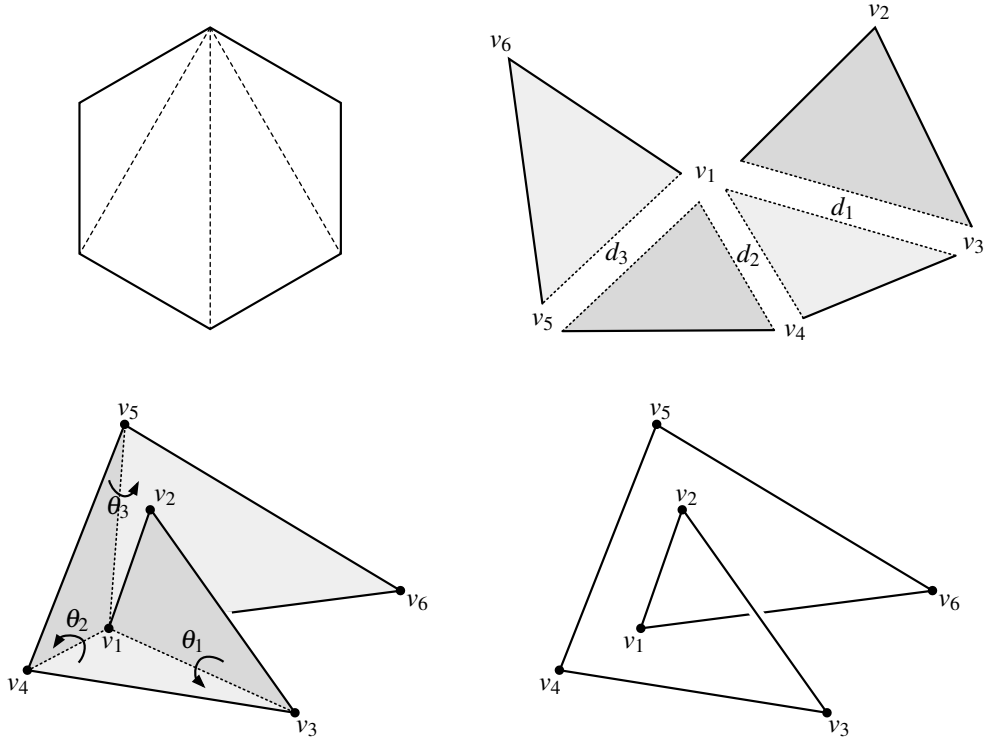

Any triangulation of an $n$-gon consists of exactly $n-3$ chordal segments and produces $n-2$ triangles. It is helpful to visualize an $n$-gon in space as the boundary of a piecewise-linear surface whose faces are exactly these triangles. Each factor of the $(n-3)$-torus acts by bending the surface around one of the chords and the fact that the chords do not intersect ensures that these actions commute. 

The conserved quantities of this action are exactly the lengths  $d_1, \dots , d_{n-3}$ of the chords, and the conjugate variables $\theta_1, \dots , \theta_{n-3}$ are the dihedral angels between the two triangles meeting at a given chord. These are the \emph{action-angle coordinates} induced by the Hamiltonian torus action. When one of the $d_i = 0$, neither the circle action nor the corresponding $\theta_i$ make sense, emphasizing that the torus action and the action-angle coordinates only make sense on the open, dense subset $\Pol(n)^\circ$ for which all $d_i >0$.

There are no restrictions on the $\theta_i$, but, since the edges of the original $n$-gon are all length 1 and $d_i$ is an edge of two different triangles, the $d_i$ satisfy a collection of triangle inequalities which can be simplified as
\begin{equation}
0 \leq d_1 \leq 2 
\qquad 
\begin{matrix} 
1 \leq d_i + d_{i+1} \\
-1 \leq d_i - d_{i+1} \leq 1 
\end{matrix}
\qquad
0 \leq d_{n-3} \leq 2,
\label{eq:fan polytope}
\end{equation} 
where the inequalities in the middle column apply for $i=1,\dots,n-4$.

Let $\mathcal{P}_n \subset \R^{n-3}$ be the convex polytope determined by the inequalities in \eqref{eq:fan polytope}. Given ${(d_1, \dots , d_{n-3})\in \mathcal{P}_n}$ and any $(n-3)$-tuple of angles $(\theta_1, \dots , \theta_{n-3})$ -- which we can think of as a point in the $(n-3)$-torus $T^{n-3}$ -- we can build a unique element of $\Pol(n)$ as illustrated in \autoref{fig:fan triangulation}, first by building the $n-2$ triangles of the triangulation, and then gluing them together according to the $\theta_i$. The main result of~\cite{Cantarella:2016iy} is that, from a measure-theoretic perspective, $\mathcal{P}_n \times T^{n-3}$ and $\Pol(n)$ are equivalent:

\begin{theorem}[Cantarella and Shonkwiler~\cite{Cantarella:2016iy}]\label{thm:sampling}
	With respect to the uniform measure on $\mathcal{P}_n \times T^{n-3}$ and the natural measure on $\Pol(n)$, the reconstruction map $\mathcal{P}_n \times T^{n-3} \to \Pol(n)$ defining action-angle coordinates is measure-preserving.
\end{theorem}

In particular this implies that any algorithm for randomly sampling points from $\mathcal{P}_n \times T^{n-3}$ gives an algorithm for sampling random equilateral $n$-gons. Since it is easy to generate independent random angles, and hence to sample $T^{n-3}$, the only challenge is to sample the convex polytope $\mathcal{P}_n$. While the paper~\cite{Cantarella:2016bt} gives an algorithm for directly sampling $\mathcal{P}_n$, this approach has not been generalized to the case of confined polygons which we are interested in, so instead we will use the hit-and-run Markov chain introduced by Boneh and Golan~\cite{Boneh:1979uy} and Smith~\cite{Smith:1984vz} for sampling arbitrary convex polytopes (see also Andersen and Diaconis' excellent survey~\cite{Andersen:2007vn}). The idea of hit-and-run is simple: given $\vec{p} \in \mathcal{P}_n$, choose a random direction $\vec{v}$; the intersection of the line through $\vec{p}$ containing the direction $\vec{v}$ with $\mathcal{P}_n$ is a connected line segment from which we choose the next point in the Markov chain uniformly at random. See \autoref{fig:hit and run}.

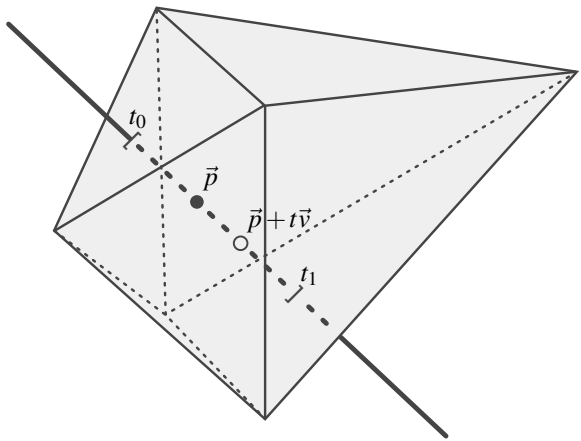
\begin{figure}[t]
	\centering
	\begingroup
	\setlength{\unitlength}{3in}
    \begin{picture}(1,0.75)
      \put(0,0){\includegraphics[width=\unitlength]{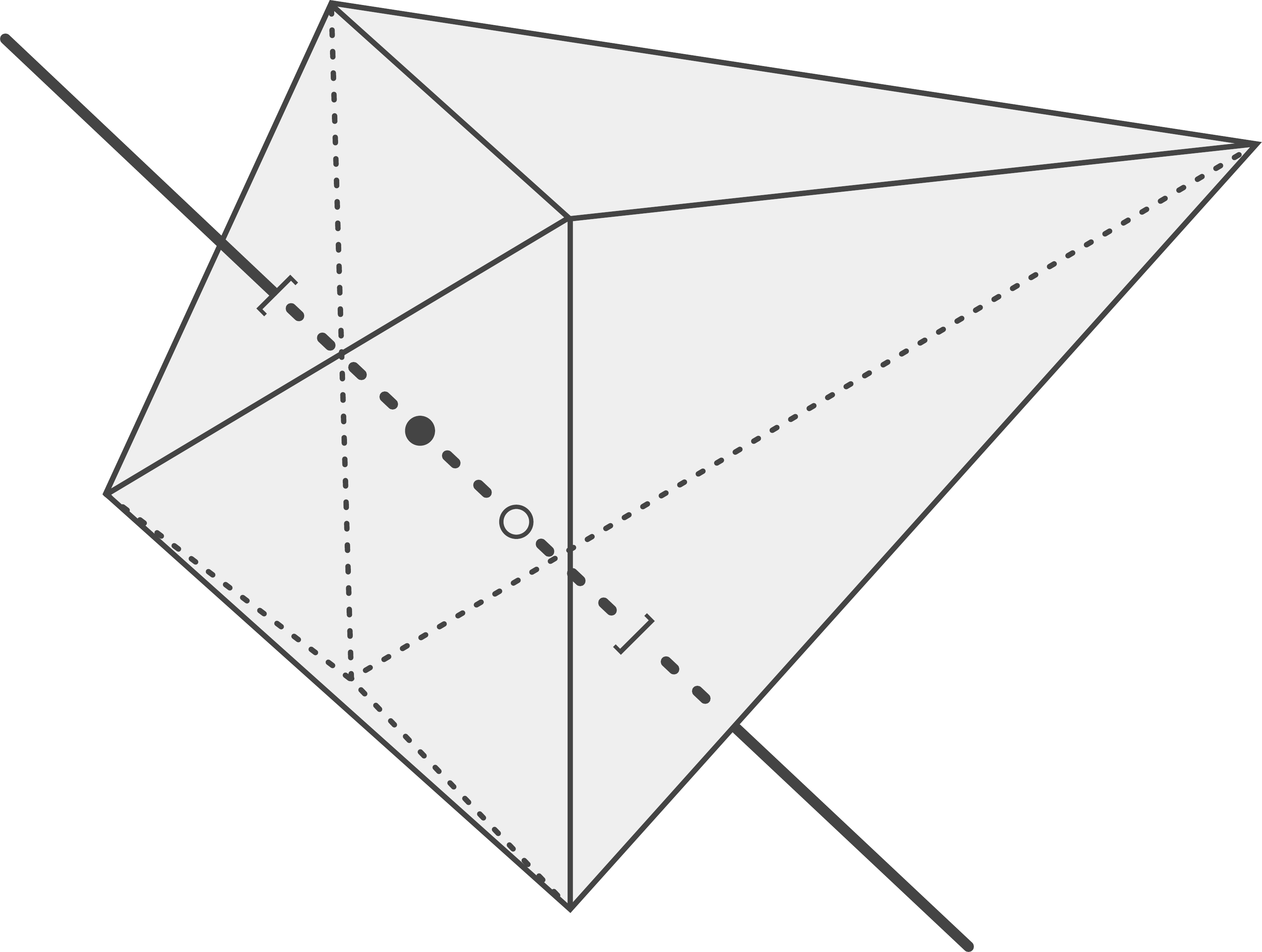}}
	  \put(0.343,0.4425){\smash{$\vec{p}$}}
	  \put(0.42,0.37){\smash{$\vec{p}+t \vec{v}$}}
	  \put(0.215,0.55){\smash{$t_0$}}
	  \put(0.515,0.275){\smash{$t_1$}}
    \end{picture}
	\endgroup
	\caption{A single hit-and-run step starting from a point $\vec{p} \in \mathcal{P}_6$ ($\bullet$). Choosing a random direction produces a unique line through $\vec{p}$ which intersects $\mathcal{P}_6$ at the points $\vec{p}+t_0 \vec{v}$ and $\vec{p}+t_1 \vec{v}$. Choosing $t$ uniformly on $[t_0,t_1]$ produces the next step $\vec{p} + t \vec{v}$ ($\circ$) in the Markov chain.}
	\label{fig:hit and run}
\end{figure}

To get an ergodic Markov chain on $\mathcal{P}_n \times T^{n-3}$ (and hence on $\Pol(n)$), we will mix hit-and-run steps on $\mathcal{P}_n$ and random samples from $T^{n-3}$ as follows. At each step, flip a coin: if heads, iterate hit-and-run 10 times on $\mathcal{P}_n$;\footnote{At least in high dimensions, a single hit-and-run step is unlikely to move very far, so it is preferable to do several steps at once, though not too many since hit-and-run steps are much more expensive than sampling from the torus. Our experience is that 10 steps provides a good balance.} if tails, randomly sample a new point $(\theta_1, \dots , \theta_{n-3})$ from $T^{n-3}$. This algorithm was dubbed \emph{Toric Symplectic Markov Chain Monte Carlo} (TSMCMC) and proved to be ergodic in~\cite{Cantarella:2016iy}.

Unfortunately, knots are quite rare among $n$-gons when $n$ is small.\footnote{The best theoretical result along these lines is due to Hake~\cite{Hake:2019cz}, who proved that the probability that a random equilateral hexagon is knotted is bounded above by $\frac{14-3\pi}{192}<\frac{1}{42}$; empirically, the fraction of nontrivial knots among random hexagons is close to $\frac{1}{10,000}$.} For example, we generated 1~billion random equilateral 10-gons using the TSMCMC procedure described above and only found 6,315,897 (0.63\%) nontrivial knots, including only three samples which were 10-crossing prime knots; given the ergodicity of TSMCMC, this gives a decent estimate for the true probability of these knots. For comparison, we know from \autoref{tab:stick numbers} that among the 10-crossing knots there are at least 71 with equilateral stick number $\leq 10$.

Since more complicated knots tend to be more condensed~\cite{Orlandini:1999hr,Moore:2004ds,Diao:2014ib,Diao:2014ck}, we expect that sampling polygons in \emph{rooted spherical confinement} -- meaning that the entire polygon is required to be contained in a small ball centered at a given root vertex -- will boost the probability of encountering complicated knots. To that end, let $\Pol(n;R)$ be the subset of $n$-gons contained in a sphere of radius $R$ centered at the root vertex. When $R \geq \left\lfloor \frac{n}{2} \right\rfloor$ we have $\Pol(n;R) = \Pol(n)$, but for $R$ very close to its minimum value of 1 (since the second vertex is always at distance 1 from the root vertex) we expect $\Pol(n;R)$ to have a much higher fraction of complicated knots than $\Pol(n)$. Indeed, as we will see, sampling 10-gons in very tight confinement produces about 7.1\% nontrivial knots, and approximately one prime knot with 10 or more crossings for every 1.9 million samples.

Fortunately, the TSMCMC algorithm described above can easily be adapted to give an ergodic Markov chain on $\Pol(n;R)$ for any $R$. The constraint that an $n$-gon lies in a sphere of radius $R$ centered at the root vertex simply means that each $d_i \leq R$. Adding these inequalities to~\eqref{eq:fan polytope} produces the polytope $\mathcal{P}_n(R) \subset \R^{n-3}$, and the analog of \autoref{thm:sampling} says that $\mathcal{P}_n(R) \times T^{n-3}$ and $\Pol(n;R)$ are measure-equivalent. Therefore, the TSMCMC algorithm described above, when applied to $\mathcal{P}_n(R)$ rather than $\mathcal{P}_n$, yields an ergodic Markov chain on $\Pol(n;R)$. Here is pseudo-code for this algorithm, taking as input the current position $(\vec{p},\vec{\theta}) \in \mathcal{P}_n(R) \times T^{n-3}$ with parameters $0< \beta < 1$ and $\gamma \in \mathbb{N}$:\footnote{As described above, we use $\beta=\frac{1}{2}$, so that hit-and-run steps and torus samples are equally likely, and $\gamma=10$, so that we iterate hit-and-run 10 times whenever we choose to do a hit-and-run step.}

			\begin{algorithm}[H]\begin{algorithmic}
				\Function{TSMCMC}{$\vec{p},~\vec{\theta}, \beta, \gamma$} 
					\State $prob = \Call{Uniform-Random-Variate}{0,1}$
					\If{$prob < \beta$} 
						\For{$i=1$ \textbf{to} $\gamma$}
							\State $\vec{v} = \Call{Random-Direction-In-Dimension}{n-3}$
							\State $(t_0, t_1) = \Call{Find-Intersection-Endpoints}{\mathcal{P}_n(R), \vec{p}, \vec{v}}$
							\State $t = \Call{Uniform-Random-Variate}{t_0, t_1}$
							\State $\vec{p} = \vec{p} + t\vec{v}$
						\EndFor
					\Else
						\For{$i=1$ \textbf{to} $n-3$}
							\State $\theta_i = \Call{Uniform-Random-Variate}{-\pi,\pi}$
						\EndFor
					\EndIf
					\State\Return $(\vec{p},~\vec{\theta})$
				\EndFunction
			\end{algorithmic}\end{algorithm}

This algorithm depends on the functions {\scshape Uniform-Random-Variate}$(a,b)$, which produces a (pseudo-)random number uniformly between $a$ and $b$, {\scshape Random-Direction-In-Dimension}$(d)$, which produces a vector uniformly at random on the unit sphere in $\R^d$, and {\scshape Find-Intersection-Endpoints}$(P,\vec{p},\vec{v})$, which produces the endpoints of the intersection of the convex polytope $P$ with the line through $\vec{p} \in P$ in the direction $\vec{v}$.

For our experiments we used the version of this algorithm implemented in Ashton, Cantarella, and Chapman's free and open-source C library {\tt plcurve}~\cite{plcurve}.

\section{Identifying Knot Types}\label{sec:pipeline}

TSMCMC gives us an efficient algorithm for generating random equilateral polygons in confinement, so our task is in principle simple: generate vast quantities of, say, 10-gons, compute their knot types, and check to see if any are knots not previously known to have (equilateral) stick number less than 11. Therefore, we need a fast, reliable strategy for determining the knot types of hundreds of billions of polygons.


Our identification pipeline begins by using the classifier contained in {\tt plcurve}. While fast, this method of classification has significant shortcomings. Identification is based solely on the HOMFLY polynomial of the given knot, even though this invariant does not uniquely identify knots (e.g.\ $5_1$ and $10_{132}$ have the same HOMFLY polynomial). Furthermore, the {\tt plcurve} classifier will only identify up to 10-crossing knots, making it unsuitable to detect many knots which we expect to generate. This method of identification maintains an important place in our classification pipeline, however, because it is sufficient to uniquely identify roughly one half of knots with crossing number 10 or less (including the trivial knot, which represents the vast majority of stick knots generated) and it performs this identification very efficiently. 

Knots which cannot be definitively classified by {\tt plcurve} require more robust identification. Toward this end, we integrated the package {\tt pyknotid}~\cite{pyknotid} as a follow-up step in our classification pipeline. This package contains a database consisting of many invariants calculated on prime knots with crossing number 15 and below. Identification can be done by computing invariants for a given knot and then checking against the database. In particular, we used {\tt pyknotid} to identify knots based on their HOMFLY polynomial, hyperbolic volume, and minimum number of crossings. 

HOMFLY polynomials were calculated by {\tt plcurve} and hyperbolic volume by {\tt SnapPy}~\cite{snappy} (via {\tt pyknotid}). A maximum crossing number was obtained by observing the number of crossings in a default projection of a given knot. Since the {\tt pyknotid} database only contains the hyperbolic volumes of knots up to 11 crossings, we augmented the database with hyperbolic volumes of knots up to 15 crossings taken from {\tt SnapPy}. While even this more robust check does not guarantee definitive identification (e.g.\ mutant knot pairs have the same HOMFLY and hyperbolic volume) it is sufficient for every knot up to 10 crossings\footnote{Under the reasonable assumption that, for knots constructed from so few segments, we are unlikely to encounter a knot with $\geq 16$ crossings and the same HOMFLY polynomial \emph{and} hyperbolic volume as one of the knots up to 10 crossings we are interested in.} and many higher-crossing knots.

It is important to note that the database queries used in {\tt pyknotid} identification can be prohibitively slow when identifying hundreds of billions of knots. One critical improvement that we made was to manually add a database index to the columns of the {\tt pyknotid} invariant database which we were using. This data structure improves query times significantly at a cost of additional storage space and more expensive database writes. However, since the invariant database is static and relatively small, these tradeoffs were well worth the additional speed. 

We verified this pipeline by plugging in all 249 of Rawdon and Scharein's minimal equilateral stick knot examples, obtained from {\tt KnotPlot}. Up to some historical labeling confusion,\footnote{``Unfortunately, the diagrams of $10_{83}$ and $10_{86}$ given by Rolfsen do not correspond to their Conway notation or Alexander polynomial---they are interchanged''~\cite{Hartley:1983hp}. This was only corrected in the 2003 edition of Rolfsen's book -- after Rawdon and Scharein's paper -- so it's not surprising that our pipeline identified Rawdon and Scharein's $10_{83}$ as a $10_{86}$ and identified their $10_{86}$ as a $10_{83}$. This is an issue that researchers in the broader computational knot theory community should be aware of: for example, the invariants computed by {\tt SnapPy} for $10_{83}$ match the invariants in the Knot Atlas~\cite{knotatlas}, {\tt pyknotid}, and KnotInfo~\cite{knotinfo} for $10_{86}$, and conversely the {\tt SnapPy} invariants for $10_{86}$ agree with the Knot Atlas, {\tt pyknotid}, and KnotInfo invariants for $10_{83}$, so whether you identify a given knot as $10_{83}$ or $10_{86}$ may depend on which software you use. We have used the Knot Atlas naming convention in this paper.} our identification agreed with theirs, verifying that our pipeline is capable of correctly identifying all knots up to 10 crossings.

In a very few cases, the above procedure was not able to uniquely identify knot type. When it failed, {\tt pyknotid} typically returned a list of possibilities; for example, 39 of our 11-gons were identified as either the Kinoshita--Terasaka knot ($11n_{34}$) or the Conway knot ($11n_{42}$), but, since these knots are mutants and hence have the same HOMFLY polynomial and hyperbolic volume, they could not be distinguished. In such cases, we computed Dowker--Thistlethwaite codes using {\tt KnotPlot}~\cite{knotplot} and then plugged these into the KnotFinder function on KnotInfo~\cite{knotinfo}, which in turn is based on Knotscape~\cite{knotscape}. There were only 889 polygons that fell into this group and KnotFinder was able to identify most, while we identified a few others in \emph{ad hoc} fashion: in the end, there were only 59 polygons, all of them 12-gons, for which we could not determine the knot type out of the 220 billion polygons we generated. It seems likely that many of these are knots with $\geq 16$ crossings, which we would not expect to be able to identify systematically.

We also used this process for verification: we loaded the coordinates of all knots mentioned in Theorems~\ref{thm:stick numbers},~\ref{thm:new bounds}, and \ref{thm:superbridge} into {\tt KnotPlot}, computed DT codes, and then plugged those DT codes into KnotInfo's KnotFinder. This was not redundant since none of these knots were originally identified by the step in the previous paragraph. Nonetheless, this verification is not entirely independent of our primary {\tt plCurve}/{\tt pyknotid}/{\tt SnapPy} pipeline since some of KnotInfo's data seems to come from the Knot Atlas~\cite{knotatlas}, which serves as the basis for {\tt pyknotid}'s database, but it was nonetheless reassuring that in each case this second approach verified the knot type identification.

The code used to generate and identify each stick knot, along with a summary of the data generated for this paper, is available as part of our {\tt stick-knot-gen} project~\cite{stick-knot-gen}.

\section{Results}\label{sec:results}

We now have all the necessary tools in place to generate large ensembles of equilateral polygons in very tight confinement and to identify their knot types. Given the imprecision of floating point arithmetic, the ``equilateral polygons'' we generate on the computer will only be approximately equilateral, so we need some guarantee that, if we generate an approximately equilateral polygon of a given knot type, there exists a true equilateral polygon of the same knot type. The following theorem does just that:

\begin{theorem}[Millett and Rawdon~\cite{Millett:2003kl}] \label{theorem:millett_rawdon}
	Let $K$ be an $n$-stick knot where $L_i$ is the length of the $i$th edge. Let $\mu(K)$ denote the minimum distance between any two non-adjacent edges. If
	\[ 
		| L_i - 1 | ~<~ \min \left\{ \frac{\mu(K)}{n}, ~ \frac{\mu(K)^2}{4} \right\} 
	\]
for all $1 \leq i \leq n$, then there exists an unit-edge equilateral stick knot equivalent to $K$. 
\end{theorem}

\begin{figure}[t]
	\centering
	\begin{minipage}{.35\textwidth}
		\includegraphics[scale=.45]{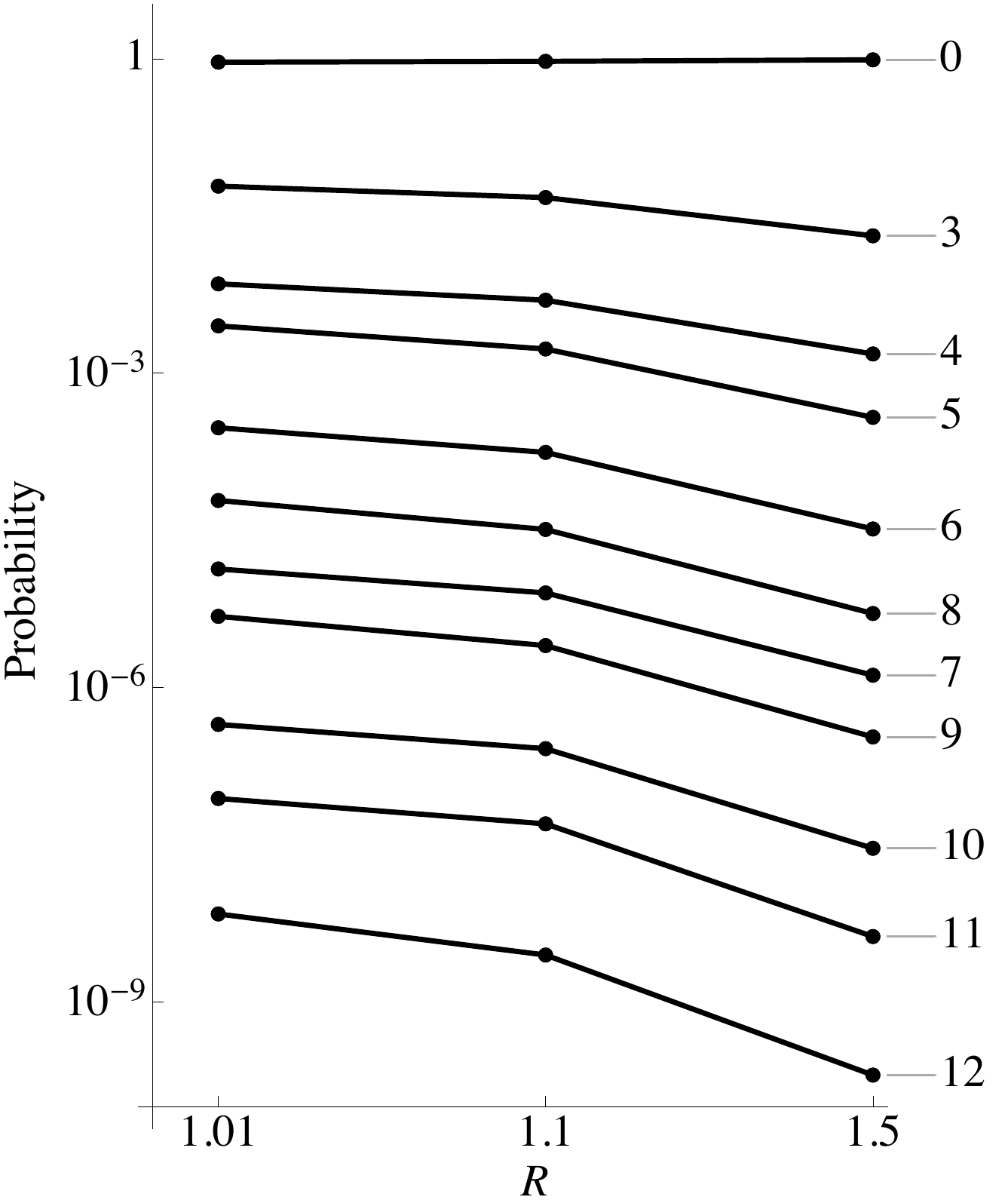} 
	\end{minipage}\hfill
	\begin{minipage}{.55\textwidth}
		\footnotesize \begin{tabular}[b]{crrr}
		 Crossings & $R=1.01$ & $R=1.1$ & $R=1.5$ \\
		 \midrule
		 0 & \text{46,435,788,199} & \text{9,457,101,359} & \text{9,776,082,950} \\
		 3 &  \text{3,045,498,555} & \text{473,763,655} & \text{204,448,942} \\
		 4 &  \text{355,759,490} & \text{49,671,263} & \text{15,247,185} \\
		 5 &  \text{141,328,728} & \text{17,008,560} & \text{3,784,326} \\
		 6 &  \text{15,054,589} & \text{1,753,019} & \text{325,635} \\
		 7 &  \text{676,177} & \text{79,926} & \text{13,097} \\
		 8 &  \text{3,046,297} & \text{322,455} & \text{50,839} \\
		 9 &  \text{238,874} & \text{25,119} & \text{3,374} \\
		 10 & \text{22,211} & \text{2,608} & 292 \\
		 11 & \text{4,364} & 500 & 42 \\
		 12 & 345 & 28 & 2 \\
		 13 & 85 & 7 & 1 \\
		 14 & 0 & 0 & 0 \\
		 15 & 1 & 1 & 0 \\
		 \midrule
		 \text{Composite} & \text{2,582,085} & \text{271,500} & \text{43,315} \\
		 \toprule
		 \text{Total} & \text{50,000,000,000} & \text{10,000,000,000} & \text{10,000,000,000} \\
		\end{tabular}
	\end{minipage}
	\caption{Left: A log plot of the probability of different crossing numbers among random 10-gons sampled from rooted spherical confinement of radii $R=1.01,1.1,1.5$. Only prime knots are included in this plot. Note the inversion between 7- and 8-crossing knots, which is mostly due to the knots $8_{19}$ and $8_{20}$. These knots both have stick number 8, whereas all 7-crossing knots have stick number 9, and each of these knots individually appeared more than 1.5 times as often as all 7-crossing knots put together at all three confinement radii. Right: The complete data for 10-gons. Keep in mind that we generated 5 times as many 10-gons at $R=1.01$ as at the other two radii. Complete frequency data can be downloaded in CSV format from the {\tt stick-knot-gen} project~\cite{stick-knot-gen}.}
	\label{fig:confinement}
\end{figure}

We focused on generating 9-, 10-, and 11-gons, since most stick numbers of knots up to 10 crossings lie in this range. Since we are free to choose the confinement radius $R$, we started by generating 10 billion $n$-gons for $n=9,10,11$ in rooted spherical confinement of radii $R=1.01, 1.1, 1.5$ to get a sense for which regime produced the highest concentration of complicated knots. Not surprisingly, $R=1.01$ was the clear winner; see \autoref{fig:confinement}. Consequently, we focused the bulk of our effort on generating polygons in this very tight confinement regime, producing a total of 50 billion each of equilateral 9-, 10-, and 11-gons in confinement radius $R=1.01$. We also generated 10 billion equilateral 12-gons in confinement radius 1.01 in the (justified!) hope that we would find 12-stick representatives of any remaining knots.

Among all these random equilateral polygons, a tiny fraction realized a knot type with fewer equal-length sticks than previously observed. For each such knot type, we took the sample maximizing the minimal distance $\mu$ between non-adjacent edges, and then stochastically applied crankshaft rotations~\cite{Millett:1994fo} (which preserve edge lengths) in search of polygons of the same knot type maximizing $\mu$. The coordinates of the vertices of each of the resulting optimal polygons are recorded in the \texttt{data} directory of {\tt stick-knot-gen}~\cite{stick-knot-gen}. For every conformation,
\[
	| L_i - 1 | ~~<~~ 10^{-2.69} \min \left\{ \frac{\mu(K)}{n}, ~ \frac{\mu(K)^2}{4} \right\} ,
\]
so the hypothesis of \autoref{theorem:millett_rawdon} is easily satisfied, and we conclude that there is a true equilateral polygon with the same number of sticks representing each knot type we found.

In particular, since we observed equilateral 9-gon versions of the knots $9_{35}$, $9_{39}$, $9_{43}$, $9_{45}$, and $9_{48}$, \autoref{thm:8sticks} and the following discussion implies \autoref{thm:stick numbers}, which we restate:

\begin{sticknumber}
	The stick number and equilateral stick number of each of the knots $9_{35}$, $9_{39}$, $9_{43}$, $9_{45}$, and $9_{48}$ is exactly 9.
\end{sticknumber}

We recall a result of Randell~\cite{Randell:1998td} relating stick number and the \emph{superbridge index} $\superbridge(K)$ of a knot $K$, which is the minimum over all embeddings of $K$ of the maximum number of local maxima in any projection to a line:

\begin{theorem}[Randell~\cite{Randell:1998td}]\label{thm:superbridge bound}
	For any knot $K$, $\superbridge(K) \leq \frac{1}{2}\stick(K)$.
\end{theorem}

Combined with \autoref{thm:stick numbers}, this implies that each of the knots $9_{35}$, $9_{39}$, $9_{43}$, $9_{45}$, and $9_{48}$ has superbridge index $\leq 4$. Since the superbridge index of a knot is strictly greater than the bridge index~\cite{Kuiper:1987ki} and each of these knots has bridge index equal to 3~\cite{knotinfo}, we get the following corollary:

\begin{corollary}\label{cor:superbridge}
	Each of the knots $9_{35}$, $9_{39}$, $9_{43}$, $9_{45}$, and $9_{48}$ has superbridge index equal to 4.
\end{corollary}

We also discovered a number of knots realizable with 10, 11, or 12 equilateral sticks that had not previously been observed:

\begin{newbounds}
	The equilateral stick number of each of the knots $9_2$, $9_3$, $9_{11}$, $9_{15}$, $9_{21}$, $9_{25}$, $9_{27}$, $10_8$, $10_{16}$, $10_{17}$, $10_{56}$, $10_{83}$, $10_{85}$, $10_{90}$, $10_{91}$, $10_{94}$, $10_{103}$, $10_{105}$, $10_{106}$, $10_{107}$, $10_{110}$, $10_{111}$, $10_{112}$, $10_{115}$, $10_{117}$, $10_{118}$, $10_{119}$, $10_{126}$, $10_{131}$, $10_{133}$, $10_{137}$, $10_{138}$, $10_{142}$, $10_{143}$, $10_{147}$, $10_{148}$, $10_{149}$, $10_{153}$, and $10_{164}$ is less than or equal to 10.
	
	The equilateral stick number of each of the knots $10_3$, $10_6$, $10_7$, $10_{10}$, $10_{15}$, $10_{18}$, $10_{20}$, $10_{21}$, $10_{22}$, $10_{23}$, $10_{24}$, $10_{26}$, $10_{28}$, $10_{30}$, $10_{31}$, $10_{34}$, $10_{35}$, $10_{38}$, $10_{39}$, $10_{43}$, $10_{44}$, $10_{46}$, $10_{47}$, $10_{50}$, $10_{51}$, $10_{53}$, $10_{54}$, $10_{55}$, $10_{57}$, $10_{62}$, $10_{64}$, $10_{65}$, $10_{68}$, $10_{70}$, $10_{71}$, $10_{72}$, $10_{73}$, $10_{74}$, $10_{75}$, $10_{77}$, $10_{78}$, $10_{82}$, $10_{84}$, $10_{95}$, $10_{97}$, $10_{100}$, and $10_{101}$ is less than or equal to 11.
	
	The equilateral stick number of each of the knots $10_{76}$ and $10_{80}$ is less than or equal to 12.
	
	In particular, all knots up to 10 crossings have equilateral stick number $\leq 12$.
\end{newbounds}

In addition to the knots listed above, among the 10-stick knots we generated were 195 different prime knots with 11 or more crossings, including the 15-crossing knots $15n_{41,127}$ and $15n_{59,007}$; among the 11-stick knots we generated were 1212 different prime knots with 11 or more crossings, including eight 16-crossing knots; and among the 12-stick knots were at least 1888 different prime knots with 11 or more crossings. Coordinates for all of these knots with more than 10 crossings can be found in the supplemental data provided in {\tt stick-knot-gen}~\cite{stick-knot-gen}.

\begin{figure}[t]
	\centering
		\subfloat[$11n_{74}$]{\includegraphics[scale=.3,valign=c]{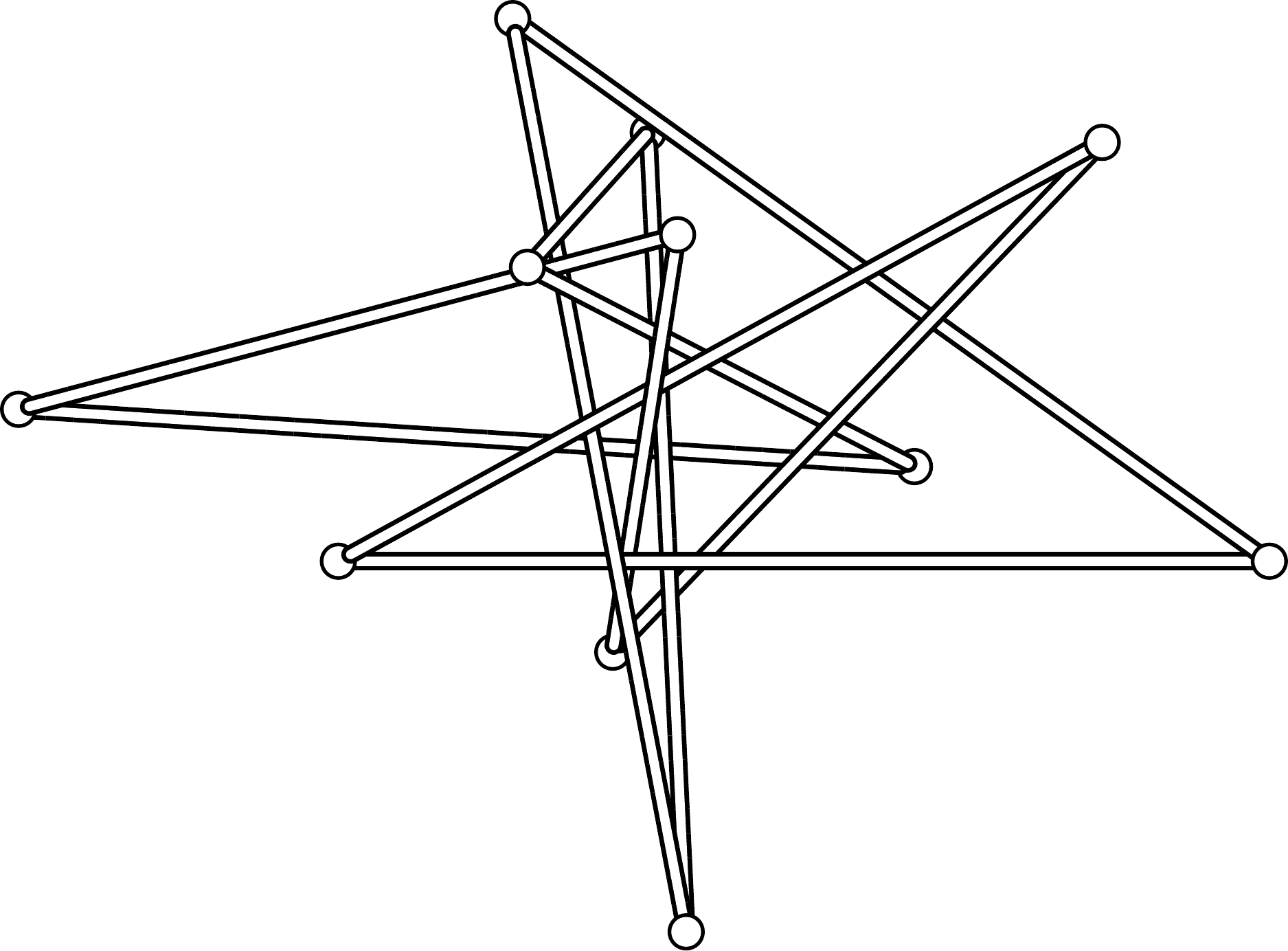}
		\vphantom{\includegraphics[scale=.3,valign=c]{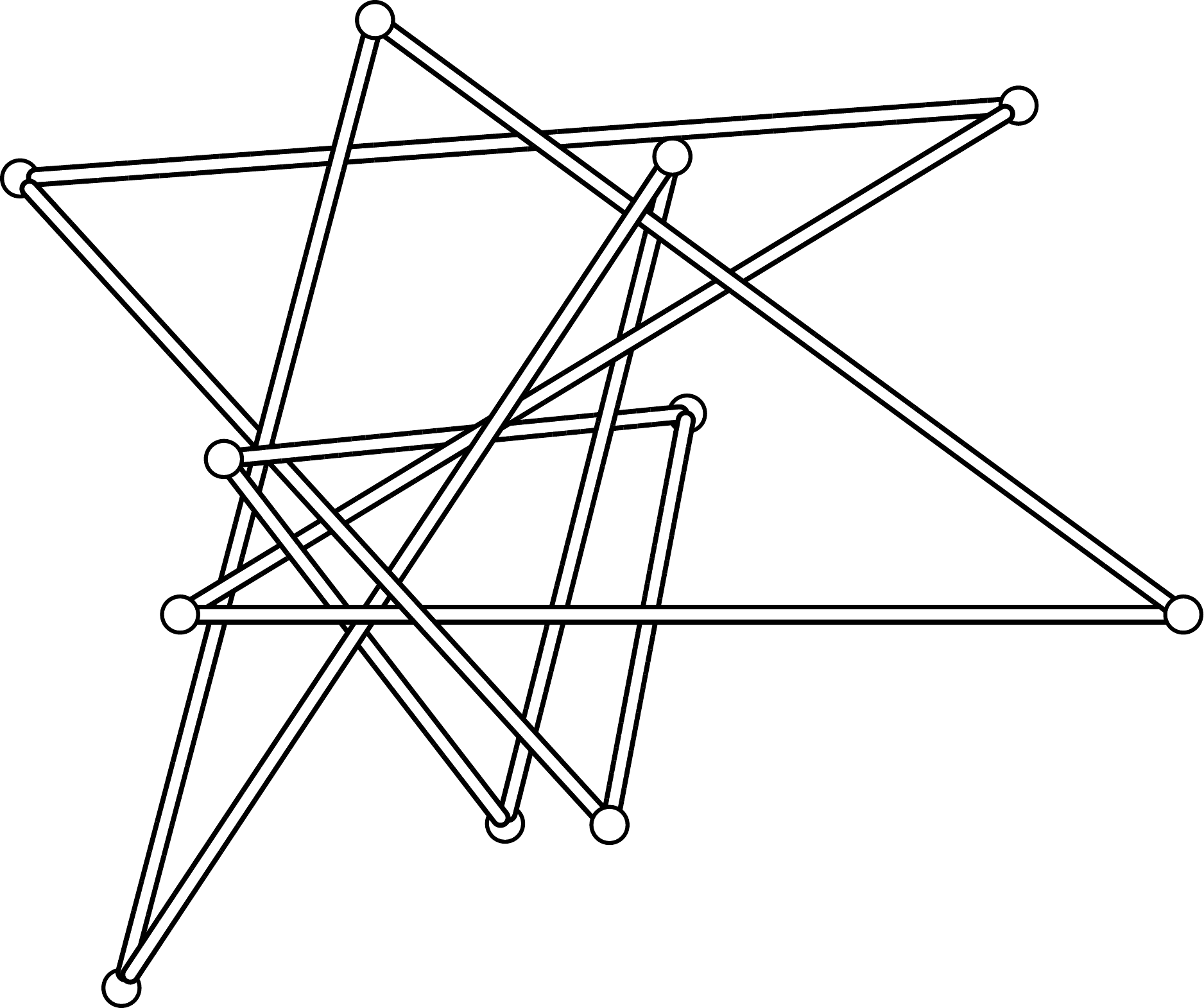}}}\qquad\quad 
		\subfloat[$11n_{81}$]{\includegraphics[scale=.3,valign=c]{K11n81z.pdf}}\qquad
	\caption{These 11-stick examples of $11n_{74}$ and $11n_{81}$ prove that these knots have superbridge index 5. Each knot is shown in orthographic perspective from the direction of the positive $z$-axis.}
	\label{fig:superbridge}
\end{figure}

We observed a number of 4-bridge knots among the 11-gons (see \autoref{fig:superbridge} for two examples), meaning these knots have superbridge index 5:

\begin{theorem}\label{thm:superbridge}
	Each of the knots $11n_{71}$, $11n_{73}$, $11n_{74}$, $11n_{75}$, $11n_{76}$, $11n_{78}$, and $11n_{81}$ has superbridge index equal to 5 and stick number either 10 or 11.
\end{theorem}

\begin{proof}
	By \autoref{thm:superbridge bound} these knots satisfy $2\superbridge(K) \leq \stick(K) \leq 11$. Since superbridge index is strictly greater than bridge index and each of these knots has bridge index 4~\cite{Musick:2012uz}, the result follows.
\end{proof}

Coordinates of each of the knots mentioned in Theorems~\ref{thm:stick numbers}, \ref{thm:new bounds}, and~\ref{thm:superbridge} are given in \hyperref[sec:coordinates]{Appendix~B} and can be downloaded from the {\tt stick-knot-gen} project~\cite{stick-knot-gen}.

\section{Conclusion and Open Questions}\label{sec:conclusion}

To the best of our knowledge the information on stick numbers of knots up to 10 crossings contained in \autoref{tab:stick numbers} represents the state of the art, but it likely remains incomplete. For example, we know from Rawdon and Scharein's work~\cite{Rawdon:2002wj} that there are equilateral 10-stick realizations of the knots $9_{16}$, $10_{109}$, $10_{113}$, and $10_{114}$ and equilateral 11-stick realizations of the knots $10_{1}$ and $10_{123}$, but we never saw any examples with these stick numbers among the 140 billion equilateral 10- and 11-gons we generated. Of course, Rawdon and Scharein's bounds on these 6 knots are still given in \autoref{tab:stick numbers}, but it seems implausible that these are the only knots achievable with 10 or 11 sticks that we failed to generate. 

We reiterate the viewpoint that generating random polygons in confinement is a strategy for boosting the probability of complicated knots. We generated equilateral stick knots in rooted spherical confinement because this is the only confinement regime in which we know of a strategy for ergodically sampling equilateral polygons, but there is no particular reason to believe that rooted spherical confinement is optimal for producing complicated knots. Aesthetically, \emph{unrooted} spherical confinement seems preferable, with the substantial side benefit that an understanding of the probability distribution of random equilateral polygons in unrooted spherical confinement would have important applications to the packing of DNA in viral capsids and other biological problems~\cite{Arsuaga:2005gr,Jardine:2006uc}.

Since we wanted to see as many complicated knots as possible, we have mostly focused on knots in very tight confinement. There is no particular reason to believe, though, that any knot which is constructible with $n$ equilateral sticks is necessarily constructible with $n$ equilateral sticks in any given confinement radius. Indeed, Rawdon and Scharein's 11-stick equilateral $10_{123}$ cannot be put in rooted spherical confinement of radius less than 1.1638, the largest of all of their examples, which may explain why we did not observe $10_{123}$ among our 11-gons. Similarly, we observed $4_1 \# \, 6_2$ and $4_1 \# \, 6_3$ knots among the 10 billion 11-stick knots in rooted spherical confinement of radius 1.1 that we generated, but not among our 50 billion 11-stick knots in confinement radius 1.01. Relatedly, composite knots remain relatively poorly understood in the context of stick number, and we know of no definitive table of stick numbers for low-crossing composite knots.

The TSMCMC approach to generating equilateral polygons described in \autoref{sec:symplectic} generalizes to arbitrary spaces of polygons with given edge lengths in rooted spherical confinement. If we want to generate $n$-gons with edgelengths given by $(r_1, \dots , r_n)$ with $r_i \geq 0$ in rooted spherical confinement of radius $R$, we must slightly modify the inequalities in~\eqref{eq:fan polytope}, but otherwise the algorithm is exactly the same and is still ergodic. 

Since equilateral polygons are more flexible than any other polygons of fixed edge lengths~\cite{Cantarella:2018tp,Khoi:2005ur}, we do not expect that sampling random polygons with other fixed edge lengths will produce generally better bounds on stick number than sampling random equilateral polygons. However, it remains an open question whether $\stick(K) = \eqstick(K)$ for all knots $K$, so sampling other fixed edge length spaces might help to shed light on this problem. We know of only two examples of knots up to 10 crossings for which the upper bounds on stick number and equilateral stick number differ: $9_{29}$ and $10_{79}$. Rawdon and Scharein's equilateral 12-stick $10_{79}$ cannot be put in rooted spherical confinement of radius less than 1.1367, second only to their $10_{123}$ mentioned above, suggesting that an equilateral 11-stick example, if it exists, may not occur in the tight confinement regime we have focused on. $9_{29}$ is particularly interesting since its stick number is exactly 9, and hence proving that the equilateral stick number of $9_{29}$ is 10 would immediately imply that stick number and equilateral stick number are distinct invariants. More generally, any theoretical results on stick number or equilateral stick number, even for individual knots, would be interesting and welcome additions to the field.

In a companion paper~\cite{withRyan}, we use the relationship between bridge index, superbridge index, and stick number to determine the exact stick number of the knots $13n_{592}$ and $15n_{41,127}$.

\section*{Acknowledgments}

We are grateful to Colin Adams, Ryan Blair, Jason Cantarella, Harrison Chapman, Tetsuo Deguchi, Kate Hake, Gyo Taek Jin, Ken Millett, Eric Rawdon, and Erica Uehara for helpful conversations about random polygons and stick knots. This work was partially supported by a grant from the Simons Foundation (\#354225, CS).

\bibliography{stickknots-special,stickknots}

\newpage

\appendix

\section{Stick Number Bounds}\label{sec:appendix}

\renewcommand{\theoremautorefname}{Thm.}

We give the best known upper bounds on stick number for all prime knots up to 10 crossings, including references for where the given bound was proved. Bounds marked with an asterisk ($\ast$) are known to be the exact stick number of the knot. Two knots, $9_{29}$ and $10_{79}$, have bounds marked with a dagger ($\dag$): these are the two knots for which the best bounds on stick number and equilateral stick number differ. In both cases, the equilateral stick number bound is 1 greater than the bound listed in this table.

Coordinates of equilateral versions of each of these knots realizing the best known bound can be downloaded either as tab-separated text files or in an SQLite database from the {\tt stick-knot-gen} project~\cite{stick-knot-gen}.

\begin{multicols*}{3}
\TrickSupertabularIntoMulticols

\tablefirsthead{
	\multicolumn{1}{c}{$K$} & \multicolumn{2}{l}{$\stick(K)$ bound} \\
	\midrule
}
\tablehead{
	\multicolumn{1}{c}{$K$} & \multicolumn{2}{l}{$\stick(K)$ bound} \\
	\midrule
}
\tablelasttail{\bottomrule}
\begin{supertabular*}{.26\textwidth}{lll} \label{tab:stick numbers}
$ 0_{1} $ & $3^*$ &  \\
$ 3_{1} $ & $6^*$ & \cite{Randell:1994bx} \\
$ 4_{1} $ & $7^*$ & \cite{Randell:1994bx} \\
$ 5_{1} $ & $8^*$ & \cite{Negami:1991gb, Meissen:1998wu} \\
$ 5_{2} $ & $8^*$ & \cite{Jin:1997da} \\
$ 6_{1} $ & $8^*$ & \cite{Negami:1991gb, Meissen:1998wu} \\
$ 6_{2} $ & $8^*$ & \cite{Negami:1991gb, Meissen:1998wu} \\
$ 6_{3} $ & $8^*$ & \cite{Negami:1991gb, Meissen:1998wu} \\
\end{supertabular*}
\end{multicols*}

\newpage

\section{Stick Knot Coordinates}\label{sec:coordinates}

We give coordinates of equilateral realizations of each of the knots mentioned in Theorems~\ref{thm:stick numbers}, \ref{thm:new bounds}, and~\ref{thm:superbridge} below. Each row gives the coordinates of a vertex. These coordinates can be downloaded either as tab-separated text files or in an SQLite database from the {\tt stick-knot-gen} project~\cite{stick-knot-gen}.

\begin{multicols*}{2}

\begin{coordinates}{ $\bm{ } $ $ \bm{9_{2}} $ }
\begin{tabular}{lll}

0.000000000000 & 0.000000000000 & 0.000000000000 \\
0.999999999973 & -0.000000000238 & -0.000000006378 \\
0.377787294690 & 0.782848228577 & -0.000000004464 \\
0.481892968809 & -0.210260817716 & 0.053818494146 \\
-0.385834243146 & 0.031557373199 & -0.380431949892 \\
-0.144253062985 & 0.496729049752 & 0.471186418115 \\
0.162700176100 & -0.186949875250 & -0.190903331479 \\
0.684883875886 & 0.311048799061 & 0.501427156130 \\
0.276972402756 & 0.908702492523 & -0.188803452286 \\
0.974395069202 & 0.223759836486 & 0.022040523880 \\

\end{tabular}
\end{coordinates}
\vspace{0.1in}

\begin{coordinates}{ $ \bm{9_{3}} $ }
\begin{tabular}{lll}

0.000000000000 & 0.000000000000 & 0.000000000000 \\
0.999999999959 & 0.000000006901 & 0.000000002872 \\
0.372624758458 & 0.778717096894 & 0.000000002526 \\
0.357966555261 & 0.557523627122 & -0.975119777801 \\
0.949956765273 & 0.586793597600 & -0.169706306290 \\
0.233474142980 & -0.087214018489 & -0.349612908264 \\
0.769117889647 & 0.327878773715 & 0.385767095931 \\
0.692196419337 & 0.577317092949 & -0.579563734569 \\
0.229065705696 & 0.855448237280 & 0.261954541320 \\
0.883542651989 & 0.355339529733 & -0.305100312661 \\

\end{tabular}
\end{coordinates}
\vspace{0.1in}

\begin{coordinates}{ $ \bm{9_{11}} $ }
\begin{tabular}{lll}

0.000000000000 & 0.000000000000 & 0.000000000000 \\
1.000000000050 & 0.000000000719 & -0.000000005639 \\
0.070051054207 & 0.367688670977 & 0.000000002547 \\
0.108161326136 & 0.376223269343 & -0.999237090356 \\
0.128417640945 & 0.525388229801 & -0.010632267228 \\
0.669249891636 & -0.315423441298 & 0.012523948856 \\
0.080331781660 & 0.286262575316 & -0.527058664982 \\
0.378578125067 & 0.387716575752 & 0.422023104114 \\
0.016208553742 & 0.340622745128 & -0.508820848917 \\
0.549670755302 & 0.812794735357 & 0.192942423967 \\

\end{tabular}
\end{coordinates}
\vspace{0.1in}

\begin{coordinates}{ $ \bm{9_{15}} $ }
\begin{tabular}{lll}

0.000000000000 & 0.000000000000 & 0.000000000000 \\
1.000000000030 & 0.000000006932 & 0.000000000155 \\
0.111494566011 & 0.458866102614 & 0.000000002644 \\
0.913398426052 & -0.121483088140 & -0.141933136704 \\
0.392655562784 & 0.073301376104 & 0.689262317371 \\
0.242928301101 & 0.098977602409 & -0.299131570962 \\
0.706185391596 & -0.095940489595 & 0.565391306402 \\
0.723960029669 & 0.053752620436 & -0.423181419823 \\
0.218381899324 & 0.417982538309 & 0.358948570026 \\
0.940617915616 & 0.177316826289 & -0.289476562011 \\

\end{tabular}
\end{coordinates}
\vspace{0.1in}

\begin{coordinates}{ $ \bm{9_{21}} $ }
\begin{tabular}{lll}

0.000000000000 & 0.000000000000 & 0.000000000000 \\
0.999999999988 & -0.000000001732 & 0.000000008667 \\
0.116912849445 & 0.469208997027 & -0.000000000434 \\
0.299142422074 & -0.499974342149 & -0.165758975821 \\
0.749383175444 & 0.375384108431 & -0.341914722302 \\
0.560880577502 & -0.293068061976 & 0.377556240605 \\
0.243435681189 & 0.217582941613 & -0.421483364398 \\
0.434924293424 & -0.518032959544 & 0.228286869895 \\
0.232932098142 & 0.459738136499 & 0.172047766022 \\
0.876539844203 & 0.055871061071 & -0.478075648873 \\

\end{tabular}
\end{coordinates}
\vspace{0.1in}

\begin{coordinates}{ $ \bm{9_{25}} $ }
\begin{tabular}{lll}

0.000000000000 & 0.000000000000 & 0.000000000000 \\
1.000000000000 & -0.000000004700 & 0.000000000423 \\
0.297955295632 & 0.712132871284 & 0.000000005690 \\
0.508816958364 & -0.259809494361 & -0.104237209739 \\
-0.248158854990 & 0.328774125004 & 0.179588337092 \\
0.699925636040 & 0.351639450261 & -0.137607146377 \\
-0.294671078076 & 0.367559159613 & -0.035020894393 \\
0.425269849180 & -0.040398159834 & 0.526455629381 \\
0.297126127816 & 0.547143528614 & -0.272527443682 \\
0.596216003363 & 0.434574530250 & 0.675034410226 \\

\end{tabular}
\end{coordinates}
\vspace{0.1in}

\begin{coordinates}{ $ \bm{9_{27}} $ }
\begin{tabular}{lll}

0.000000000000 & 0.000000000000 & 0.000000000000 \\
1.000000000110 & -0.000000004551 & 0.000000003763 \\
0.073818706276 & 0.377078516712 & 0.000000001267 \\
0.337143126688 & -0.567884771228 & 0.194176812560 \\
0.329472622438 & 0.320571578615 & -0.264720210545 \\
1.035060372380 & 0.015232248216 & 0.374743486959 \\
0.286264824178 & 0.159348536766 & -0.272199881776 \\
0.153213377745 & -0.050853560207 & 0.696362140529 \\
0.359580868476 & -0.267310983410 & -0.257869824305 \\
0.740850508636 & -0.137529788493 & 0.657439032285 \\

\end{tabular}
\end{coordinates}
\vspace{0.1in}

\begin{coordinates}{ $ \bm{9_{35}} $ }
\begin{tabular}{lll}

0.000000000000 & 0.000000000000 & 0.000000000000 \\
0.999999999935 & -0.000000004489 & -0.000000006432 \\
0.319886562581 & 0.733106885568 & -0.000000004192 \\
0.558297071061 & 0.119190194309 & -0.752506964721 \\
0.404552532087 & -0.127158422011 & 0.204402109335 \\
-0.027529078858 & 0.464069732340 & -0.476593302569 \\
0.722283839107 & 0.597133655906 & 0.171538304670 \\
0.398573339896 & -0.273154812344 & -0.199688707759 \\
0.579860399875 & 0.677042860632 & -0.453183055204 \\

\end{tabular}
\end{coordinates}
\vspace{0.1in}

\begin{coordinates}{ $ \bm{9_{39}} $ }
\begin{tabular}{lll}

0.000000000000 & 0.000000000000 & 0.000000000000 \\
1.000000000050 & 0.000000000947 & 0.000000006319 \\
0.104465710090 & 0.444992512795 & 0.000000003714 \\
0.385853182847 & -0.298247876945 & -0.606971835169 \\
0.796515213189 & 0.452556941574 & -0.089623962555 \\
0.009741230786 & -0.138634877270 & 0.087799074143 \\
0.398256777112 & 0.381736091648 & -0.672641406730 \\
0.839413704804 & 0.113763413508 & 0.183846310016 \\
0.633086696236 & -0.138145067675 & -0.761654236076 \\

\end{tabular}
\end{coordinates}
\vspace{0.1in}

\begin{coordinates}{ $ \bm{9_{43}} $ }
\begin{tabular}{lll}

0.000000000000 & 0.000000000000 & 0.000000000000 \\
0.999999999976 & 0.000000004022 & -0.000000002241 \\
0.123205019372 & 0.480864394418 & -0.000000003121 \\
0.536781299678 & -0.425023241338 & 0.091226373410 \\
0.655686472913 & 0.511849704238 & -0.237605868949 \\
0.130945474415 & -0.166080553236 & 0.277231435907 \\
0.466858434733 & 0.773364194053 & 0.345099321197 \\
0.615641998136 & 0.055456942967 & -0.334954081175 \\
0.725959223388 & 0.255055144373 & 0.638694042013 \\

\end{tabular}
\end{coordinates}
\vspace{0.1in}

\begin{coordinates}{ $ \bm{9_{45}} $ }
\begin{tabular}{lll}

0.000000000000 & 0.000000000000 & 0.000000000000 \\
0.999999999976 & -0.000000003405 & 0.000000001560 \\
0.190799788131 & 0.587532988957 & 0.000000003866 \\
0.618532229693 & -0.316173690378 & 0.018952475450 \\
-0.168249315081 & 0.200679011262 & -0.318446525138 \\
0.665648318313 & -0.150518614023 & 0.107317452947 \\
-0.041271805254 & 0.070599205069 & -0.564523933847 \\
0.448339322230 & -0.004956385151 & 0.304137275502 \\
0.754811851302 & 0.192729936955 & -0.626988229980 \\

\end{tabular}
\end{coordinates}
\vspace{0.1in}

\begin{coordinates}{ $ \bm{9_{48}} $ }
\begin{tabular}{lll}

0.000000000000 & 0.000000000000 & 0.000000000000 \\
0.999999999996 & -0.000000003564 & 0.000000005184 \\
0.248818509090 & 0.660095722600 & 0.000000004874 \\
0.485153931120 & -0.218697628419 & 0.414569437562 \\
0.660586822993 & 0.407932508387 & -0.344744708295 \\
0.723479433911 & -0.029769564714 & 0.552173021844 \\
0.032497435603 & -0.280512703542 & -0.125817948837 \\
0.550894729598 & 0.543039852974 & 0.104454568431 \\
0.814146590780 & -0.157417362979 & -0.558914217521 \\

\end{tabular}
\end{coordinates}
\vspace{0.1in}

\begin{coordinates}{ $ \bm{10_{3}} $ }
\begin{tabular}{lll}

0.000000000000 & 0.000000000000 & 0.000000000000 \\
1.000000000000 & 0.000000000000 & 0.000000000000 \\
0.295099436133 & 0.709306136348 & 0.000000000000 \\
0.743634443843 & -0.171244634576 & 0.153123109571 \\
-0.047653955306 & 0.076707113133 & -0.405788869349 \\
0.676921396905 & -0.561236049500 & -0.144984191057 \\
0.550736205384 & 0.364494842341 & 0.211526136046 \\
0.651652056677 & -0.120179196512 & -0.657327745236 \\
0.891059518228 & 0.459376243326 & 0.121645655791 \\
0.479007477232 & -0.405786741957 & -0.164197175862 \\
0.181898143352 & 0.092028546074 & -0.979001436236 \\

\end{tabular}
\end{coordinates}
\vspace{0.1in}

\begin{coordinates}{ $ \bm{10_{6}} $ }
\begin{tabular}{lll}

0.000000000000 & 0.000000000000 & 0.000000000000 \\
1.000000000000 & 0.000000000000 & 0.000000000000 \\
0.076622881037 & 0.383894121049 & 0.000000000000 \\
0.384225465337 & -0.256613828862 & 0.703654898536 \\
0.168821062775 & -0.085069515409 & -0.257684527411 \\
0.013953723959 & 0.357144939045 & 0.625753350224 \\
0.324207756843 & 0.193585699194 & -0.310724524656 \\
0.767571065652 & -0.335789629956 & 0.412594725254 \\
0.604039392965 & 0.388524122413 & -0.257201497606 \\
-0.185719292546 & 0.218009978692 & 0.332040511631 \\
0.729461040842 & 0.618973278237 & 0.291133424262 \\

\end{tabular}
\end{coordinates}
\vspace{0.1in}

\begin{coordinates}{ $ \bm{10_{7}} $ }
\begin{tabular}{lll}

0.000000000000 & 0.000000000000 & 0.000000000000 \\
1.000000000000 & 0.000000000000 & 0.000000000000 \\
0.031034387757 & 0.247195554753 & 0.000000000000 \\
0.433153690970 & -0.344914154636 & 0.698359619416 \\
0.470489324230 & -0.068757795504 & -0.262027654448 \\
0.168431571655 & 0.094843643407 & 0.677118597898 \\
0.509844174096 & 0.301821913731 & -0.239722408225 \\
0.376735408028 & -0.453231490128 & 0.402287259706 \\
0.891806573865 & 0.290193977028 & -0.024351075313 \\
0.071592959174 & -0.152060592797 & 0.338499474837 \\
0.679248606446 & 0.641953919712 & 0.355691573709 \\

\end{tabular}
\end{coordinates}
\vspace{0.1in}

\begin{coordinates}{ $ \bm{10_{8}} $ }
\begin{tabular}{lll}

0.000000000000 & 0.000000000000 & 0.000000000000 \\
1.000000000000 & 0.000000001474 & 0.000000003955 \\
0.317011378212 & 0.730429014105 & -0.000000004922 \\
0.366653340016 & -0.250196709107 & -0.189496354767 \\
0.627561312612 & 0.438817472656 & 0.486659313107 \\
0.371113287586 & 0.688873356349 & -0.446992905598 \\
0.395833015069 & 0.558378615406 & 0.544147880372 \\
0.643436615623 & 0.507929649370 & -0.423399203151 \\
0.578102914549 & 0.233720047605 & 0.536048868314 \\
0.497987557224 & 0.825317540193 & -0.266194197282 \\

\end{tabular}
\end{coordinates}
\vspace{0.1in}

\begin{coordinates}{ $ \bm{10_{10}} $ }
\begin{tabular}{lll}

0.000000000000 & 0.000000000000 & 0.000000000000 \\
1.000000000000 & 0.000000000000 & 0.000000000000 \\
0.236713421703 & 0.646060058657 & 0.000000000000 \\
0.083067046323 & 0.065888378288 & 0.799870997493 \\
0.336900556633 & -0.045168792223 & -0.160980111640 \\
-0.042583906181 & 0.314217920204 & 0.691564745487 \\
0.664727599604 & 0.406271213498 & -0.009318001467 \\
-0.181389636369 & -0.126488366052 & -0.025219365773 \\
0.623612904546 & -0.651797537767 & 0.250500026422 \\
0.441792759643 & 0.278664764841 & -0.067590115413 \\
0.139532019856 & 0.485809336725 & 0.862855783886 \\

\end{tabular}
\end{coordinates}
\vspace{0.1in}

\begin{coordinates}{ $ \bm{10_{15}} $ }
\begin{tabular}{lll}

0.000000000000 & 0.000000000000 & 0.000000000000 \\
1.000000000000 & 0.000000000000 & 0.000000000000 \\
0.113905380497 & 0.463504396191 & 0.000000000000 \\
-0.122662405858 & -0.008179859426 & -0.849440783965 \\
0.461462969379 & 0.325450558517 & -0.109516267390 \\
0.014576088792 & -0.279703585255 & -0.768364184769 \\
-0.141603970788 & 0.658328832868 & -0.459005216028 \\
0.101154465045 & -0.220358757355 & -0.047937751225 \\
0.556571132196 & 0.663683303751 & 0.057253947484 \\
-0.191330403038 & 0.099390079681 & -0.292340747648 \\
0.720579538016 & 0.413211436303 & -0.556795688114 \\

\end{tabular}
\end{coordinates}
\vspace{0.1in}

\begin{coordinates}{ $ \bm{10_{16}} $ }
\begin{tabular}{lll}

0.000000000000 & 0.000000000000 & 0.000000000000 \\
0.999999999986 & -0.000000003603 & -0.000000006635 \\
0.170126987122 & 0.557952308449 & -0.000000003966 \\
0.701043378871 & -0.131832773177 & -0.492264491984 \\
0.209092261015 & -0.077091365888 & 0.376635661043 \\
0.728773122463 & 0.437933917337 & -0.305039305390 \\
0.148392593962 & -0.375633518675 & -0.269451818753 \\
0.943392667843 & 0.140481965703 & 0.049295242929 \\
0.139867814717 & -0.088052260750 & -0.500359119158 \\
0.848537298960 & 0.512785648430 & -0.130519466055 \\

\end{tabular}
\end{coordinates}
\vspace{0.1in}

\begin{coordinates}{ $ \bm{10_{17}} $ }
\begin{tabular}{lll}

0.000000000000 & 0.000000000000 & 0.000000000000 \\
1.000000000130 & 0.000000000461 & 0.000000005147 \\
0.222791291659 & 0.629242897682 & 0.000000005194 \\
0.446486902003 & -0.318564366999 & -0.227204006466 \\
0.449121999841 & 0.603307281376 & 0.160282406381 \\
-0.205386109038 & 0.308428325255 & -0.535896835337 \\
0.252426661285 & -0.074742809820 & 0.266342124442 \\
0.674185090528 & 0.052223831190 & -0.631432512892 \\
0.480892535438 & 0.189091456676 & 0.340115345538 \\
0.346662501947 & 0.848704529161 & -0.399406724928 \\

\end{tabular}
\end{coordinates}
\vspace{0.1in}

\begin{coordinates}{ $ \bm{10_{18}} $ }
\begin{tabular}{lll}

0.000000000000 & 0.000000000000 & 0.000000000000 \\
1.000000000000 & 0.000000000000 & 0.000000000000 \\
0.312756934294 & 0.726427538464 & 0.000000000000 \\
0.680458110320 & -0.203177008275 & 0.025124327480 \\
-0.226772818215 & 0.210320007757 & -0.052026563909 \\
0.495053855569 & 0.378212968617 & -0.723426747636 \\
0.484376310065 & 0.835115589836 & 0.166025884461 \\
-0.008546411973 & 0.210754645410 & -0.439944907883 \\
0.171149599288 & 0.412684529962 & 0.522829035087 \\
0.793440265196 & 0.439084936158 & -0.259511904450 \\
-0.079857812413 & 0.843380434260 & -0.531349388731 \\

\end{tabular}
\end{coordinates}
\vspace{0.1in}

\begin{coordinates}{ $ \bm{10_{20}} $ }
\begin{tabular}{lll}

0.000000000000 & 0.000000000000 & 0.000000000000 \\
1.000000000000 & 0.000000000000 & 0.000000000000 \\
0.153219825859 & 0.531942982547 & 0.000000000000 \\
0.008535087060 & -0.404403502899 & -0.319877453960 \\
0.210642827625 & 0.525004170887 & -0.011083287707 \\
0.615746552276 & -0.208370891073 & 0.534858092635 \\
0.627026972699 & 0.312587105244 & -0.318649680141 \\
0.455379230799 & 0.376610548920 & 0.664426131463 \\
0.169400011945 & 0.387883602480 & -0.293743375800 \\
-0.000931253153 & 0.258966804641 & 0.683174082006 \\
0.858364990240 & 0.465666964496 & 0.215322599155 \\

\end{tabular}
\end{coordinates}
\vspace{0.1in}

\begin{coordinates}{ $ \bm{10_{21}} $ }
\begin{tabular}{lll}

0.000000000000 & 0.000000000000 & 0.000000000000 \\
1.000000000000 & 0.000000000000 & 0.000000000000 \\
0.152873600923 & 0.531391441395 & 0.000000000000 \\
0.677578940856 & -0.317913440094 & -0.058013140976 \\
0.107136492050 & 0.415724385633 & 0.311271238957 \\
0.615929401065 & -0.427586631806 & 0.484350133111 \\
0.416686206671 & 0.342005321812 & -0.122304937265 \\
0.181167997366 & 0.017801254357 & 0.793895311442 \\
0.587128141122 & 0.580635225517 & 0.073885393046 \\
0.179722861126 & -0.278695750243 & -0.235260546489 \\
0.948044869340 & 0.049486939149 & 0.314264170042 \\

\end{tabular}
\end{coordinates}
\vspace{0.1in}

\begin{coordinates}{ $ \bm{10_{22}} $ }
\begin{tabular}{lll}

0.000000000000 & 0.000000000000 & 0.000000000000 \\
1.000000000000 & 0.000000000000 & 0.000000000000 \\
0.153386561242 & 0.532208310076 & 0.000000000000 \\
0.220071080827 & -0.454485174753 & 0.148287362390 \\
0.249610063606 & 0.455525533908 & 0.561818447823 \\
0.194953281050 & -0.511142690599 & 0.311688122093 \\
0.512673200914 & 0.361769028037 & -0.058553939253 \\
-0.162052422628 & 0.264886162891 & 0.673128398497 \\
0.311802314667 & -0.489954816050 & 0.219612374023 \\
-0.070649665646 & 0.373085146239 & -0.110376271610 \\
0.810798874329 & 0.467353597502 & 0.352400057162 \\

\end{tabular}
\end{coordinates}
\vspace{0.1in}

\begin{coordinates}{ $ \bm{10_{23}} $ }
\begin{tabular}{lll}

0.000000000000 & 0.000000000000 & 0.000000000000 \\
1.000000000000 & 0.000000000000 & 0.000000000000 \\
0.034684780542 & 0.261087202067 & 0.000000000000 \\
0.245456758352 & -0.096049234001 & -0.909960844983 \\
0.224842233420 & -0.270024572828 & 0.074573371203 \\
-0.171071128665 & 0.399210124353 & -0.554217158155 \\
0.457491855053 & 0.043919795462 & 0.137647971545 \\
0.024410110808 & -0.286911689008 & -0.700797454041 \\
0.759148469928 & 0.126090756195 & -0.162663154916 \\
-0.123194791675 & 0.305526796885 & -0.597718411183 \\
0.533318181093 & -0.438217026749 & -0.723558950731 \\

\end{tabular}
\end{coordinates}
\vspace{0.1in}

\begin{coordinates}{ $ \bm{10_{24}} $ }
\begin{tabular}{lll}

0.000000000000 & 0.000000000000 & 0.000000000000 \\
1.000000000000 & 0.000000000000 & 0.000000000000 \\
0.059107707270 & 0.338705910018 & 0.000000000000 \\
0.197268732252 & -0.163592938267 & 0.853585026923 \\
0.152985492617 & -0.263497378168 & -0.140426088496 \\
-0.244162696994 & 0.507938799152 & 0.356725336939 \\
0.468912601385 & -0.010653396558 & -0.115064609319 \\
0.090083300454 & -0.298754768476 & 0.764415896579 \\
0.447169291507 & 0.589401686127 & 0.475162189357 \\
-0.382205058471 & 0.038132057747 & 0.384387829404 \\
0.399296901281 & -0.407039576689 & 0.821511270547 \\

\end{tabular}
\end{coordinates}
\vspace{0.1in}

\begin{coordinates}{ $ \bm{10_{26}} $ }
\begin{tabular}{lll}

0.000000000000 & 0.000000000000 & 0.000000000000 \\
1.000000000000 & 0.000000000000 & 0.000000000000 \\
0.348637772012 & 0.758766925973 & 0.000000000000 \\
0.681164467759 & 0.025483919399 & -0.593061572592 \\
-0.014455223673 & -0.089976529398 & 0.116009742338 \\
0.892906488952 & 0.207717515435 & -0.180761176274 \\
0.195642126755 & 0.078041720486 & -0.885748124674 \\
0.601863980366 & 0.726675728401 & -0.242119559875 \\
0.163601610177 & -0.164584913520 & -0.125583887057 \\
0.839511047571 & 0.544600959778 & 0.074920051950 \\
0.801330839793 & -0.235074192887 & -0.550099090197 \\

\end{tabular}
\end{coordinates}
\vspace{0.1in}

\begin{coordinates}{ $ \bm{10_{28}} $ }
\begin{tabular}{lll}

0.000000000000 & 0.000000000000 & 0.000000000000 \\
1.000000000000 & 0.000000000000 & 0.000000000000 \\
0.051376909076 & 0.316408330114 & 0.000000000000 \\
0.870985927072 & -0.090912040437 & 0.402903429315 \\
0.027249430203 & 0.381469960764 & 0.148023126443 \\
0.857445338309 & 0.234123831006 & -0.389623478873 \\
0.341180165884 & 0.670913221877 & 0.347048293195 \\
0.516155685242 & -0.294807190072 & 0.155299359508 \\
0.476422054584 & 0.630256800555 & -0.222427859316 \\
0.117043028445 & -0.046748122969 & 0.419842371269 \\
0.625838914318 & 0.767319867056 & 0.139806562602 \\

\end{tabular}
\end{coordinates}
\vspace{0.1in}

\begin{coordinates}{ $ \bm{10_{30}} $ }
\begin{tabular}{lll}

0.000000000000 & 0.000000000000 & 0.000000000000 \\
1.000000000000 & 0.000000000000 & 0.000000000000 \\
0.109193226206 & 0.454382318938 & 0.000000000000 \\
0.732880141801 & 0.053838927198 & 0.671252279438 \\
-0.158540649557 & -0.283378805427 & 0.368509528015 \\
0.504887345897 & 0.398930002775 & 0.675621861044 \\
0.848110486707 & 0.228671769972 & -0.248071816413 \\
0.720801803494 & -0.326183235607 & 0.574076845303 \\
0.501041766076 & 0.649351182008 & 0.580251471035 \\
0.524198619185 & -0.348042670866 & 0.648583472348 \\
0.765340971759 & 0.562505020285 & 0.312795938433 \\

\end{tabular}
\end{coordinates}
\vspace{0.1in}

\begin{coordinates}{ $ \bm{10_{31}} $ }
\begin{tabular}{lll}

0.000000000000 & 0.000000000000 & 0.000000000000 \\
1.000000000000 & 0.000000000000 & 0.000000000000 \\
0.156348008746 & 0.536890414939 & 0.000000000000 \\
0.738412822641 & -0.239268241351 & 0.242442349212 \\
-0.057834893835 & 0.352748022659 & 0.366966712907 \\
0.778563498045 & 0.262847525082 & -0.173732492378 \\
0.678742303281 & -0.554060333436 & 0.394332010798 \\
-0.056866672375 & -0.024747081543 & -0.028405397419 \\
0.688686616621 & 0.634932172636 & 0.066323557352 \\
0.572772251293 & -0.309173324894 & -0.242267870661 \\
0.721902248429 & 0.015211402997 & 0.691827837639 \\

\end{tabular}
\end{coordinates}
\vspace{0.1in}

\begin{coordinates}{ $ \bm{10_{34}} $ }
\begin{tabular}{lll}

0.000000000000 & 0.000000000000 & 0.000000000000 \\
1.000000000000 & 0.000000000000 & 0.000000000000 \\
0.168855351506 & 0.556056268087 & 0.000000000000 \\
0.619075796556 & -0.333404800147 & -0.078489228269 \\
-0.038820253294 & -0.273370247637 & 0.672222860416 \\
0.441374422834 & 0.005944352041 & -0.159279651562 \\
0.229197152155 & 0.812816984748 & 0.392025491164 \\
0.343861238139 & 0.121474075591 & -0.321344769881 \\
-0.257371018547 & 0.402265465390 & 0.426770173783 \\
0.572077602092 & 0.436405785749 & -0.130768546170 \\
0.851860482027 & -0.404071793259 & 0.333256215325 \\

\end{tabular}
\end{coordinates}
\vspace{0.1in}

\begin{coordinates}{ $ \bm{10_{35}} $ }
\begin{tabular}{lll}

0.000000000000 & 0.000000000000 & 0.000000000000 \\
1.000000000000 & 0.000000000000 & 0.000000000000 \\
0.063963886523 & 0.351903956026 & 0.000000000000 \\
0.801592032088 & 0.018035589819 & 0.586887240376 \\
0.745075837399 & -0.137270072922 & -0.399361242948 \\
0.420297328348 & 0.602006913136 & 0.190544222313 \\
0.891684159889 & -0.232671268801 & -0.094248312084 \\
0.543478575090 & 0.285463349641 & 0.686961895151 \\
0.196189781642 & 0.060188373012 & -0.223335681936 \\
-0.101108989012 & 0.805308228255 & 0.373665023499 \\
0.585864917390 & 0.731296751195 & -0.349238257159 \\

\end{tabular}
\end{coordinates}
\vspace{0.1in}

\begin{coordinates}{ $ \bm{10_{38}} $ }
\begin{tabular}{lll}

0.000000000000 & 0.000000000000 & 0.000000000000 \\
1.000000000000 & 0.000000000000 & 0.000000000000 \\
0.042565739497 & 0.288651410555 & 0.000000000000 \\
0.299137109616 & -0.339369508064 & -0.734684189167 \\
-0.121891088950 & -0.181290973131 & 0.158482331033 \\
0.465703686834 & 0.570153289201 & -0.141624150202 \\
0.256282770549 & -0.258286230760 & 0.377828293809 \\
-0.042612773453 & 0.289957461635 & -0.403254485111 \\
0.428953348474 & 0.054489027589 & 0.446557264172 \\
0.181580887646 & 0.091383272752 & -0.521660521400 \\
0.987327776594 & 0.143232232619 & 0.068325610905 \\

\end{tabular}
\end{coordinates}
\vspace{0.1in}

\begin{coordinates}{ $ \bm{10_{39}} $ }
\begin{tabular}{lll}

0.000000000000 & 0.000000000000 & 0.000000000000 \\
1.000000000000 & 0.000000000000 & 0.000000000000 \\
0.386085532640 & 0.789372552580 & 0.000000000000 \\
0.655247224790 & 0.103454295781 & 0.676068138924 \\
0.179942376189 & 0.296837814552 & -0.182237235382 \\
0.768318812581 & -0.481583876430 & 0.036561393742 \\
0.542580893886 & 0.345796955049 & 0.550840842839 \\
0.654607311786 & 0.581087379980 & -0.414606461589 \\
0.751378847683 & -0.137966689974 & 0.273577027258 \\
0.538737337463 & 0.128544713730 & -0.666505554129 \\
0.738927652261 & 0.654482042335 & 0.160122393766 \\

\end{tabular}
\end{coordinates}
\vspace{0.1in}

\begin{coordinates}{ $ \bm{10_{43}} $ }
\begin{tabular}{lll}

0.000000000000 & 0.000000000000 & 0.000000000000 \\
1.000000000000 & 0.000000000000 & 0.000000000000 \\
0.318437803844 & 0.731760188020 & 0.000000000000 \\
0.672194137225 & -0.077028770377 & 0.469805148298 \\
0.680538390430 & 0.296589141463 & -0.457740014694 \\
-0.008159373750 & -0.063659999024 & 0.171478505689 \\
0.680010984596 & 0.097180484521 & -0.536018417801 \\
0.563485984663 & 0.527426984116 & 0.359140715524 \\
0.087887770996 & 0.609917918156 & -0.516645551186 \\
0.096445407122 & 0.536177876618 & 0.480595228735 \\
0.907761080012 & 0.409683164561 & -0.090163885730 \\

\end{tabular}
\end{coordinates}
\vspace{0.1in}

\begin{coordinates}{ $ \bm{10_{44}} $ }
\begin{tabular}{lll}

0.000000000000 & 0.000000000000 & 0.000000000000 \\
1.000000000000 & 0.000000000000 & 0.000000000000 \\
0.323360674893 & 0.736314622780 & 0.000000000000 \\
0.641681674809 & -0.200752320489 & -0.143447847126 \\
0.639301667753 & 0.365101155387 & 0.681054535772 \\
0.488951302441 & 0.379130726291 & -0.307478690170 \\
0.346594381494 & -0.335388236923 & 0.377500987095 \\
0.093691083172 & 0.618535273508 & 0.216039357148 \\
0.791760526519 & -0.095711259443 & 0.165532501838 \\
-0.116379869987 & 0.299489387437 & 0.027338909581 \\
0.698080235015 & 0.390063083406 & 0.600445481659 \\

\end{tabular}
\end{coordinates}
\vspace{0.1in}

\begin{coordinates}{ $ \bm{10_{46}} $ }
\begin{tabular}{lll}

0.000000000000 & 0.000000000000 & 0.000000000000 \\
0.999999999950 & 0.000000003250 & -0.000000004981 \\
0.059838434901 & 0.340728974007 & 0.000000001238 \\
0.264703535774 & -0.605763013880 & 0.249365610990 \\
-0.340287133056 & 0.190224934726 & 0.229630456974 \\
0.533012665099 & -0.141461219398 & -0.127205300341 \\
0.066351402945 & 0.564557987927 & 0.405489870750 \\
-0.610515719252 & -0.154874930051 & 0.249710618618 \\
0.219725008840 & 0.333722310426 & -0.018567352869 \\
0.099849053615 & -0.247293907615 & 0.786448118883 \\
-0.579207847879 & 0.476112861900 & 0.661690873211 \\

\end{tabular}
\end{coordinates}
\vspace{0.1in}

\begin{coordinates}{ $ \bm{10_{47}} $ }
\begin{tabular}{lll}

0.000000000000 & 0.000000000000 & 0.000000000000 \\
1.000000000000 & 0.000000000000 & 0.000000000000 \\
0.165588160362 & 0.551141435452 & 0.000000000000 \\
0.095759985280 & -0.446338809330 & 0.012537433338 \\
-0.220196478441 & 0.314826430734 & 0.578928632380 \\
0.158765337554 & -0.175707043551 & -0.205778225385 \\
-0.182817874754 & -0.013164147260 & 0.719911093951 \\
-0.080856412279 & -0.386021037549 & -0.202358714475 \\
0.469085936002 & 0.383696117709 & 0.121830912857 \\
-0.078145737358 & -0.443236130962 & -0.007477059747 \\
0.335806360494 & -0.302821186027 & 0.891926800553 \\

\end{tabular}
\end{coordinates}
\vspace{0.1in}

\begin{coordinates}{ $ \bm{10_{50}} $ }
\begin{tabular}{lll}

0.000000000000 & 0.000000000000 & 0.000000000000 \\
1.000000000000 & 0.000000000000 & 0.000000000000 \\
0.253721842446 & 0.665634217538 & 0.000000000000 \\
0.397863922767 & -0.286613247305 & 0.269161338939 \\
0.321083854063 & 0.539467277052 & -0.289135883270 \\
0.348192874372 & 0.213848756850 & 0.655976639300 \\
0.896641043206 & 0.371395788823 & -0.165231944934 \\
0.189078087270 & -0.055272372930 & 0.398069877480 \\
0.793570755186 & -0.022904445784 & -0.397882967687 \\
0.603591791023 & 0.134421688756 & 0.571217894229 \\
0.681996607669 & 0.683164171035 & -0.261088763722 \\

\end{tabular}
\end{coordinates}
\vspace{0.1in}

\begin{coordinates}{ $ \bm{10_{51}} $ }
\begin{tabular}{lll}

0.000000000000 & 0.000000000000 & 0.000000000000 \\
1.000000000000 & 0.000000000000 & 0.000000000000 \\
0.049409076176 & 0.310446284473 & 0.000000000000 \\
0.339700457654 & -0.628330208939 & 0.185552173944 \\
0.396469223771 & 0.319603553402 & 0.498919194238 \\
0.260027147177 & -0.311797647655 & -0.264439229466 \\
0.794481832371 & 0.372414587639 & 0.231758116908 \\
-0.016088756293 & 0.190814133043 & -0.325015262201 \\
0.396550190684 & -0.116543419028 & 0.532458022878 \\
0.393816451236 & 0.110986781676 & -0.441309164101 \\
0.907739460953 & 0.061031437454 & 0.415071360939 \\

\end{tabular}
\end{coordinates}
\vspace{0.1in}

\begin{coordinates}{ $ \bm{10_{53}} $ }
\begin{tabular}{lll}

0.000000000000 & 0.000000000000 & 0.000000000000 \\
1.000000000000 & 0.000000000000 & 0.000000000000 \\
0.276696289187 & 0.690530044187 & 0.000000000000 \\
0.451929295519 & -0.129456850950 & -0.544898967971 \\
-0.059707330161 & 0.228484127024 & 0.236193996758 \\
0.689486128597 & 0.611408191327 & -0.304248712816 \\
0.663822845236 & 0.042622055329 & 0.517836275117 \\
0.368345820209 & 0.387923768824 & -0.372927472652 \\
0.133716609227 & 0.090371109389 & 0.552497590882 \\
0.218995904543 & -0.031244197766 & -0.436409461806 \\
0.832934251017 & 0.543356538493 & 0.104805560730 \\

\end{tabular}
\end{coordinates}
\vspace{0.1in}

\begin{coordinates}{ $ \bm{10_{54}} $ }
\begin{tabular}{lll}

0.000000000000 & 0.000000000000 & 0.000000000000 \\
1.000000000000 & 0.000000000000 & 0.000000000000 \\
0.221183803503 & 0.627252207708 & 0.000000000000 \\
0.520072074242 & -0.101264131820 & -0.616384413054 \\
0.473743137298 & 0.527954695476 & 0.159461768207 \\
0.497640642336 & -0.419203540785 & -0.160413496130 \\
-0.287482522153 & 0.083789162154 & 0.200944992593 \\
0.437316592413 & 0.119159358336 & -0.487106744376 \\
0.548273456766 & 0.005907990559 & 0.500244612512 \\
0.080465783207 & 0.088631520853 & -0.379705839556 \\
0.850788843275 & 0.517333187268 & 0.092329396777 \\

\end{tabular}
\end{coordinates}
\vspace{0.1in}

\begin{coordinates}{ $ \bm{10_{55}} $ }
\begin{tabular}{lll}

0.000000000000 & 0.000000000000 & 0.000000000000 \\
1.000000000000 & 0.000000000000 & 0.000000000000 \\
0.199666949075 & 0.599555675143 & 0.000000000000 \\
0.268371087704 & -0.390398687539 & 0.123572250690 \\
0.178671037992 & 0.576187196356 & 0.363708941213 \\
0.608529357965 & 0.328533991447 & -0.504559286265 \\
0.049347909438 & 0.047585527107 & 0.275430500853 \\
0.980330143272 & 0.219956911967 & -0.046377183903 \\
0.193407754679 & 0.077371167883 & 0.553974762804 \\
0.136966560724 & 0.081534709932 & -0.444422481053 \\
0.865713543322 & 0.489238100104 & 0.105764560776 \\

\end{tabular}
\end{coordinates}
\vspace{0.1in}

\begin{coordinates}{ $ \bm{10_{56}} $ }
\begin{tabular}{lll}

0.000000000000 & 0.000000000000 & 0.000000000000 \\
0.999999999994 & 0.000000000073 & -0.000000007842 \\
0.271802905621 & 0.685367778513 & -0.000000005822 \\
0.896016993509 & -0.048752654511 & 0.267252613068 \\
0.291988109173 & 0.036843296874 & -0.525099859006 \\
0.597998386570 & 0.509785978725 & 0.301146429828 \\
0.341767748426 & -0.344996126576 & -0.150177636121 \\
0.733245678519 & 0.544861944839 & 0.084124824280 \\
0.305854993332 & -0.168305130290 & -0.471509880489 \\
0.961293419491 & 0.203252418628 & 0.186019934291 \\

\end{tabular}
\end{coordinates}
\vspace{0.1in}

\begin{coordinates}{ $ \bm{10_{57}} $ }
\begin{tabular}{lll}

0.000000000000 & 0.000000000000 & 0.000000000000 \\
1.000000000000 & 0.000000000000 & 0.000000000000 \\
0.171261006519 & 0.559635310433 & 0.000000000000 \\
0.221330891175 & -0.409697710940 & 0.240596135309 \\
0.946832562250 & 0.210156725822 & -0.058448681556 \\
0.440464566780 & -0.125987853159 & 0.735653500273 \\
0.270241247687 & 0.736075086598 & 0.258294610615 \\
0.794041777033 & -0.070799044284 & -0.014811856484 \\
0.388005154758 & 0.834544561505 & 0.109635790458 \\
0.473496873073 & -0.032126562535 & -0.381864289762 \\
0.835892278112 & 0.060778793421 & 0.545518136879 \\

\end{tabular}
\end{coordinates}
\vspace{0.1in}

\begin{coordinates}{ $ \bm{10_{62}} $ }
\begin{tabular}{lll}

0.000000000000 & 0.000000000000 & 0.000000000000 \\
1.000000000000 & 0.000000000000 & 0.000000000000 \\
0.124676953427 & 0.483538585986 & 0.000000000000 \\
0.631577671352 & -0.366696275412 & -0.141958946990 \\
0.612948715135 & 0.517207535368 & 0.325338619106 \\
0.063345944985 & 0.041445070104 & -0.361383210815 \\
0.839214730355 & 0.307858430972 & 0.210500998654 \\
0.739881432097 & -0.613776150889 & -0.164629103279 \\
0.795497425539 & 0.356684540680 & 0.070132280668 \\
-0.088380286250 & 0.013428934314 & -0.247571636999 \\
0.809411973147 & 0.424422322448 & -0.405854592107 \\

\end{tabular}
\end{coordinates}
\vspace{0.1in}

\begin{coordinates}{ $ \bm{10_{64}} $ }
\begin{tabular}{lll}

0.000000000000 & 0.000000000000 & 0.000000000000 \\
1.000000000000 & 0.000000000000 & 0.000000000000 \\
0.251048347121 & 0.662624646123 & 0.000000000000 \\
0.630407466821 & 0.034864587088 & 0.679708736578 \\
0.842967773749 & -0.031791224429 & -0.295163110748 \\
0.343447444812 & 0.315308739049 & 0.498562947998 \\
0.804666185458 & -0.471528110644 & 0.088483632204 \\
0.785658175010 & 0.516377058138 & -0.065406175875 \\
0.001570041164 & 0.072133002392 & 0.368013963592 \\
0.849610092629 & -0.360007983508 & 0.061285680878 \\
0.734744707556 & 0.461614005304 & -0.497054046182 \\

\end{tabular}
\end{coordinates}
\vspace{0.1in}

\begin{coordinates}{ $ \bm{10_{65}} $ }
\begin{tabular}{lll}

0.000000000000 & 0.000000000000 & 0.000000000000 \\
1.000000000000 & 0.000000000000 & 0.000000000000 \\
0.063224786261 & 0.349931706085 & 0.000000000000 \\
0.685162684008 & -0.322313422506 & 0.401596487076 \\
0.000499294950 & 0.369628399698 & 0.630621722300 \\
0.438156017741 & -0.158516562208 & -0.097059177689 \\
0.342301884633 & 0.785382440869 & 0.218957684452 \\
0.618470427390 & 0.033875748015 & -0.380182221236 \\
0.253783283300 & 0.009816902387 & 0.550637019478 \\
0.165007358441 & 0.090796587319 & -0.442117293429 \\
0.375380168224 & 0.924317636282 & 0.068750538638 \\

\end{tabular}
\end{coordinates}
\vspace{0.1in}

\begin{coordinates}{ $ \bm{10_{68}} $ }
\begin{tabular}{lll}

0.000000000000 & 0.000000000000 & 0.000000000000 \\
1.000000000000 & 0.000000000000 & 0.000000000000 \\
0.347320929587 & 0.757634497000 & 0.000000000000 \\
0.276386325646 & -0.198816193063 & 0.283143708109 \\
0.675222034860 & 0.660323221093 & -0.037495565998 \\
0.001265461952 & -0.078330797540 & -0.024351038979 \\
0.205100741712 & 0.887931942761 & -0.181791489403 \\
0.874894080168 & 0.175504785253 & 0.027551360974 \\
0.149468760102 & -0.483176123425 & -0.172192388741 \\
0.413839481288 & 0.193048749221 & 0.515432540558 \\
0.715688024165 & 0.608631637607 & -0.342575804415 \\

\end{tabular}
\end{coordinates}
\vspace{0.1in}

\begin{coordinates}{ $ \bm{10_{70}} $ }
\begin{tabular}{lll}

0.000000000000 & 0.000000000000 & 0.000000000000 \\
1.000000000000 & 0.000000000000 & 0.000000000000 \\
0.159287462737 & 0.541481698388 & 0.000000000000 \\
0.452980170800 & -0.407397467797 & -0.115641347325 \\
0.826205879375 & 0.500576208292 & 0.074848479145 \\
0.178660657168 & -0.226415521254 & -0.153555133716 \\
0.632121283963 & 0.567901660268 & 0.250714176379 \\
0.421142134367 & -0.211656922267 & -0.339011365573 \\
0.430485352212 & -0.395737114228 & 0.643855455114 \\
0.572728317667 & 0.300201439049 & -0.060017025412 \\
0.725661396865 & -0.624070869241 & 0.289743140151 \\

\end{tabular}
\end{coordinates}
\vspace{0.1in}

\begin{coordinates}{ $ \bm{10_{71}} $ }
\begin{tabular}{lll}

0.000000000000 & 0.000000000000 & 0.000000000000 \\
1.000000000000 & 0.000000000000 & 0.000000000000 \\
0.320387658575 & 0.733571445316 & 0.000000000000 \\
0.061607418801 & -0.228451970033 & -0.086854682219 \\
0.073695281875 & 0.630171526178 & -0.599318900795 \\
0.337235356429 & -0.025593900744 & 0.108154302908 \\
0.190251955558 & 0.126427253942 & -0.869232748569 \\
0.462896515470 & 0.195970416569 & 0.090365444677 \\
0.377459321184 & -0.083872033296 & -0.865871288086 \\
-0.095281371072 & 0.578946955475 & -0.285193041542 \\
0.576145940286 & 0.060865960197 & -0.815077413735 \\

\end{tabular}
\end{coordinates}
\vspace{0.1in}

\begin{coordinates}{ $ \bm{10_{72}} $ }
\begin{tabular}{lll}

0.000000000000 & 0.000000000000 & 0.000000000000 \\
1.000000000000 & 0.000000000000 & 0.000000000000 \\
0.165249950050 & 0.550629053091 & 0.000000000000 \\
0.134341040454 & -0.370863523257 & 0.387164139664 \\
0.344216223103 & 0.515466930978 & -0.025595758585 \\
0.224923060434 & -0.477390601099 & -0.027345712500 \\
0.301793145935 & 0.377596312811 & -0.540267116809 \\
0.631438555737 & 0.025544685627 & 0.335742903450 \\
0.141893325042 & -0.127424166521 & -0.522712685021 \\
0.255788774769 & 0.129110558425 & 0.437088211743 \\
0.966486350143 & -0.148821355799 & -0.209180159303 \\

\end{tabular}
\end{coordinates}
\vspace{0.1in}

\begin{coordinates}{ $ \bm{10_{73}} $ }
\begin{tabular}{lll}

0.000000000000 & 0.000000000000 & 0.000000000000 \\
1.000000000000 & 0.000000000000 & 0.000000000000 \\
0.294030726755 & 0.708242462179 & 0.000000000000 \\
-0.040292843672 & -0.112361349038 & 0.463505269951 \\
0.385063169336 & 0.161789012373 & -0.398999130928 \\
0.621589001899 & 0.331686300054 & 0.557656733095 \\
0.103279296536 & -0.336085495014 & 0.023394738340 \\
0.570968175228 & 0.506827061175 & 0.289412661358 \\
0.788478965487 & -0.327139892096 & -0.217724575465 \\
0.072485772395 & 0.083270797899 & 0.347002626805 \\
0.918457182255 & 0.384420587271 & -0.093044163958 \\

\end{tabular}
\end{coordinates}
\vspace{0.1in}

\begin{coordinates}{ $ \bm{10_{74}} $ }
\begin{tabular}{lll}

0.000000000000 & 0.000000000000 & 0.000000000000 \\
1.000000000000 & 0.000000000000 & 0.000000000000 \\
0.133248925597 & 0.498740989914 & 0.000000000000 \\
0.335744002161 & -0.478570205367 & -0.062117401307 \\
0.619835157814 & 0.392683069636 & 0.338144945256 \\
0.216342481549 & 0.561083420138 & -0.561207589711 \\
-0.123417681982 & -0.311173387669 & -0.209453282685 \\
0.187908041660 & 0.638344828993 & -0.170833986660 \\
0.636271222835 & -0.115606020717 & 0.309299925899 \\
-0.091677425042 & -0.088230828608 & -0.375785006728 \\
0.399059444251 & 0.775701252373 & -0.488916278130 \\

\end{tabular}
\end{coordinates}
\vspace{0.1in}

\begin{coordinates}{ $ \bm{10_{75}} $ }
\begin{tabular}{lll}

0.000000000000 & 0.000000000000 & 0.000000000000 \\
1.000000000000 & 0.000000000000 & 0.000000000000 \\
0.050315486428 & 0.313208117203 & 0.000000000000 \\
0.225290547043 & -0.670010332096 & -0.051625663387 \\
0.169150861530 & 0.327689430601 & -0.013631989927 \\
0.021479422664 & -0.656163252935 & 0.087496854827 \\
0.460006847701 & 0.184606121615 & -0.229994178604 \\
0.642885993679 & -0.115692392653 & 0.706155392936 \\
0.150043869612 & 0.004541574585 & -0.155616290207 \\
-0.061524555805 & 0.810839963562 & 0.396761940085 \\
0.693051205175 & 0.677688900625 & -0.245800286770 \\

\end{tabular}
\end{coordinates}
\vspace{0.1in}

\begin{coordinates}{ $ \bm{10_{76}} $ }
\begin{tabular}{lll}

0.000000000000 & 0.000000000000 & 0.000000000000 \\
1.000000000000 & 0.000000000000 & 0.000000000000 \\
0.269460969040 & 0.682870942597 & 0.000000000000 \\
0.681493819513 & -0.156850220478 & 0.353690964563 \\
0.445649503757 & 0.174933158011 & -0.559707770099 \\
0.284902902020 & 0.331702508253 & 0.414758186645 \\
0.917817382758 & -0.370766701000 & 0.089248939275 \\
0.439353448103 & 0.054455564261 & -0.679033748518 \\
0.484110031835 & 0.319920463021 & 0.284047367349 \\
-0.197776501470 & -0.392483252610 & 0.118180849777 \\
0.412934547688 & 0.325827276418 & -0.215078806233 \\
0.666061840025 & -0.621153418738 & -0.412952849188 \\

\end{tabular}
\end{coordinates}
\vspace{0.1in}

\begin{coordinates}{ $ \bm{10_{77}} $ }
\begin{tabular}{lll}

0.000000000000 & 0.000000000000 & 0.000000000000 \\
1.000000000000 & 0.000000000000 & 0.000000000000 \\
0.052334747758 & 0.319265672588 & 0.000000000000 \\
0.774227807783 & -0.360666122829 & 0.128697954406 \\
0.336638211828 & 0.490036612793 & 0.419937033372 \\
0.414121017525 & -0.280340969793 & -0.212925349038 \\
0.521765476270 & 0.585504450990 & 0.275668938495 \\
0.125870818218 & -0.184813704668 & 0.775546282239 \\
0.511685180828 & 0.459330009343 & 0.115072094020 \\
0.499764614909 & -0.530978821989 & 0.253442317309 \\
0.767175753726 & 0.400182219866 & 0.501293879674 \\

\end{tabular}
\end{coordinates}
\vspace{0.1in}

\begin{coordinates}{ $ \bm{10_{78}} $ }
\begin{tabular}{lll}

0.000000000000 & 0.000000000000 & 0.000000000000 \\
1.000000000000 & 0.000000000000 & 0.000000000000 \\
0.070288147928 & 0.368287757217 & 0.000000000000 \\
0.699058903187 & -0.401401557222 & -0.110569871884 \\
0.005747872289 & 0.243069583261 & -0.433024153170 \\
0.700528902036 & -0.147369180618 & 0.170993141586 \\
0.546131555945 & 0.559483440642 & -0.519311741301 \\
0.684486011823 & -0.425352188249 & -0.414636870370 \\
0.669143108953 & 0.522017619035 & -0.094863444994 \\
0.003012542748 & -0.153909529148 & -0.410129299655 \\
0.953669347070 & 0.124761362721 & -0.273768841965 \\

\end{tabular}
\end{coordinates}
\vspace{0.1in}

\begin{coordinates}{ $ \bm{10_{80}} $ }
\begin{tabular}{lll}

0.000000000000 & 0.000000000000 & 0.000000000000 \\
1.000000000000 & 0.000000000000 & 0.000000000000 \\
0.077354829353 & 0.385649956672 & 0.000000000000 \\
0.244184999440 & -0.511364422062 & -0.409307828771 \\
0.001804587929 & 0.149368247640 & 0.301103372791 \\
0.359734830005 & 0.554751622636 & -0.540056695855 \\
0.524085657887 & 0.045289575674 & 0.304595438505 \\
0.436140339957 & 0.349938004253 & -0.643800552639 \\
-0.279302715325 & -0.058302576359 & -0.076807728180 \\
0.523575210542 & 0.526120097062 & -0.194439250086 \\
0.120927309214 & -0.191268381118 & -0.762971077631 \\
0.260963093274 & 0.779722443955 & -0.569149518441 \\

\end{tabular}
\end{coordinates}
\vspace{0.1in}

\begin{coordinates}{ $ \bm{10_{82}} $ }
\begin{tabular}{lll}

0.000000000000 & 0.000000000000 & 0.000000000000 \\
1.000000000000 & 0.000000000000 & 0.000000000000 \\
0.077146569142 & 0.385151327602 & 0.000000000000 \\
0.286570519522 & -0.579860030704 & -0.157780503698 \\
0.293218234768 & 0.405844280066 & -0.326133762090 \\
-0.019530278575 & 0.116092535428 & 0.578428172705 \\
0.568015326702 & 0.395371444654 & -0.181041018627 \\
-0.356275256745 & 0.060622595283 & 0.002344163970 \\
0.542186776667 & 0.385067012819 & 0.298151186277 \\
-0.018532061843 & -0.131674695214 & -0.348820522006 \\
0.797501592192 & 0.443145714047 & -0.409405772521 \\

\end{tabular}
\end{coordinates}
\vspace{0.1in}

\begin{coordinates}{ $ \bm{10_{83}} $ }
\begin{tabular}{lll}

0.000000000000 & 0.000000000000 & 0.000000000000 \\
0.999999999980 & 0.000000005093 & -0.000000001684 \\
0.122715518548 & 0.479970773617 & -0.000000001823 \\
0.861747440976 & -0.185421680459 & 0.105283898764 \\
-0.065080481830 & -0.170257378812 & -0.269896118642 \\
0.652297079341 & -0.020256510557 & 0.410448774932 \\
0.484533122215 & 0.162803758130 & -0.558232902765 \\
0.311527446458 & -0.396806783177 & 0.252263906767 \\
0.772866828445 & 0.423913000913 & -0.084759847082 \\
0.823814591027 & -0.547995335277 & 0.145019419946 \\

\end{tabular}
\end{coordinates}
\vspace{0.1in}

\begin{coordinates}{ $ \bm{10_{84}} $ }
\begin{tabular}{lll}

0.000000000000 & 0.000000000000 & 0.000000000000 \\
1.000000000000 & 0.000000000000 & 0.000000000000 \\
0.170768676687 & 0.558905548762 & 0.000000000000 \\
0.423598764798 & 0.164295556110 & 0.883379816526 \\
0.647761944004 & -0.088570920756 & -0.057794669935 \\
0.359461924174 & 0.840816057460 & 0.172689479817 \\
0.181166107691 & -0.114631632232 & -0.062534485982 \\
0.299931734920 & 0.625219376513 & 0.599670324215 \\
0.568765118194 & 0.362872389495 & -0.327099696039 \\
-0.029038735794 & -0.008028033601 & 0.383578450737 \\
0.714958576275 & 0.648754304127 & 0.260676211203 \\

\end{tabular}
\end{coordinates}
\vspace{0.1in}

\begin{coordinates}{ $ \bm{10_{85}} $ }
\begin{tabular}{lll}

0.000000000000 & 0.000000000000 & 0.000000000000 \\
1.000000000060 & 0.000000002791 & 0.000000001638 \\
0.070688083933 & 0.369295768590 & -0.000000002148 \\
0.276859498358 & -0.356657570556 & 0.656113629188 \\
0.315375235740 & 0.023136229494 & -0.268155383273 \\
0.036854896889 & 0.376746969923 & 0.624809266128 \\
0.379430467193 & -0.359146755126 & 0.040769537966 \\
0.314032040088 & 0.634023966641 & -0.055848075322 \\
-0.164618271214 & -0.197823892065 & 0.225084335543 \\
0.168109390351 & 0.518543228219 & 0.838362781475 \\

\end{tabular}
\end{coordinates}
\vspace{0.1in}

\begin{coordinates}{ $ \bm{10_{90}} $ }
\begin{tabular}{lll}

0.000000000000 & 0.000000000000 & 0.000000000000 \\
0.999999999994 & -0.000000000360 & 0.000000003321 \\
0.337563715215 & 0.749118260415 & -0.000000003672 \\
0.665676915476 & -0.068142238483 & 0.473737270324 \\
0.897129492854 & 0.147535026059 & -0.474900185187 \\
0.473191483963 & 0.424650244573 & 0.387354835249 \\
0.577171909229 & -0.145193735227 & -0.427792948780 \\
0.690048756191 & -0.017005224727 & 0.557512345619 \\
-0.006452210629 & 0.561569222517 & 0.133085659098 \\
0.949451260310 & 0.313042280070 & -0.023384507929 \\

\end{tabular}
\end{coordinates}
\vspace{0.1in}

\begin{coordinates}{ $ \bm{10_{91}} $ }
\begin{tabular}{lll}

0.000000000000 & 0.000000000000 & 0.000000000000 \\
1.000000000000 & -0.000000003242 & -0.000000008185 \\
0.005769707732 & 0.107266608324 & -0.000000000125 \\
0.795084073560 & -0.183704940982 & -0.540664766327 \\
0.433260594849 & 0.139341393286 & 0.333820703070 \\
0.305630130519 & 0.199428011314 & -0.656179328707 \\
0.643948149307 & -0.291838559445 & 0.146440179303 \\
0.232335253477 & 0.606122360365 & 0.302135428143 \\
0.367197681911 & -0.196291371259 & -0.279194365276 \\
0.341799321375 & 0.768532668589 & -0.540861129356 \\

\end{tabular}
\end{coordinates}
\vspace{0.1in}

\begin{coordinates}{ $ \bm{10_{94}} $ }
\begin{tabular}{lll}

0.000000000000 & 0.000000000000 & 0.000000000000 \\
1.000000000020 & 0.000000002720 & -0.000000003949 \\
0.210135095219 & 0.613280878155 & 0.000000004339 \\
0.526748406461 & -0.052196857784 & 0.675940381975 \\
0.455019031302 & 0.363491023189 & -0.230734026154 \\
0.505124603514 & -0.420826783510 & 0.387598416577 \\
0.246492577522 & 0.544974638969 & 0.405958359132 \\
0.825736760075 & -0.033194856768 & -0.168668660053 \\
0.175479987160 & 0.014593573085 & 0.589541334099 \\
0.925175949959 & 0.377795489491 & 0.036332213368 \\

\end{tabular}
\end{coordinates}
\vspace{0.1in}

\begin{coordinates}{ $ \bm{10_{95}} $ }
\begin{tabular}{lll}

0.000000000000 & 0.000000000000 & 0.000000000000 \\
1.000000000000 & 0.000000000000 & 0.000000000000 \\
0.314323185864 & 0.727906111086 & 0.000000000000 \\
0.158546572160 & -0.063852132939 & -0.590637392689 \\
0.822703347529 & 0.177715976660 & 0.116851565639 \\
0.124138078041 & 0.440549449130 & -0.548674639717 \\
0.341600062394 & 0.340734857275 & 0.422277126652 \\
0.069596495056 & -0.278582906643 & -0.314241418023 \\
0.700324158866 & 0.497352332484 & -0.303891597472 \\
0.229222278409 & 0.053735090583 & 0.458516928791 \\
0.487161224639 & 0.871516389063 & -0.055974322713 \\

\end{tabular}
\end{coordinates}
\vspace{0.1in}

\begin{coordinates}{ $ \bm{10_{97}} $ }
\begin{tabular}{lll}

0.000000000000 & 0.000000000000 & 0.000000000000 \\
1.000000000000 & 0.000000000000 & 0.000000000000 \\
0.386903099172 & 0.790007715275 & 0.000000000000 \\
0.600379257457 & -0.177323892366 & -0.136738768100 \\
0.306689800522 & 0.610562893525 & 0.404538347159 \\
0.408929697235 & 0.315639640322 & -0.545497061781 \\
-0.101905450829 & -0.173632077250 & 0.161369713422 \\
0.321954778613 & 0.706387662422 & -0.052891257555 \\
0.738070699603 & -0.153173441689 & 0.243760471915 \\
0.301994825091 & 0.281530481696 & -0.544192782194 \\
0.724738681049 & 0.622447505886 & 0.295487642055 \\

\end{tabular}
\end{coordinates}
\vspace{0.1in}

\begin{coordinates}{ $ \bm{10_{100}} $ }
\begin{tabular}{lll}

0.000000000000 & 0.000000000000 & 0.000000000000 \\
1.000000000000 & 0.000000000000 & 0.000000000000 \\
0.240940477857 & 0.651021229949 & 0.000000000000 \\
-0.012780851898 & -0.303057351908 & -0.159246809923 \\
0.893604728975 & 0.119377906566 & -0.162938853372 \\
0.073907572706 & 0.276811459533 & 0.387798151382 \\
0.141128105649 & 0.268416249824 & -0.609904670306 \\
0.601831660991 & -0.174512617149 & 0.159228772535 \\
-0.120384019016 & 0.489865555459 & 0.351598612463 \\
0.304336208105 & -0.167989233760 & -0.270365859976 \\
0.821315590484 & -0.007364171067 & 0.570426568291 \\

\end{tabular}
\end{coordinates}
\vspace{0.1in}

\begin{coordinates}{ $ \bm{10_{101}} $ }
\begin{tabular}{lll}

0.000000000000 & 0.000000000000 & 0.000000000000 \\
1.000000000000 & 0.000000000000 & 0.000000000000 \\
0.121037379708 & 0.476890671045 & 0.000000000000 \\
0.673848032866 & -0.075270303972 & -0.624114283944 \\
0.703035373024 & 0.266915297190 & 0.315064673195 \\
0.139531953501 & -0.555579367930 & 0.237822055195 \\
0.628442824865 & 0.110511846542 & -0.325461757696 \\
0.418304825606 & 0.404397596343 & 0.606993703616 \\
0.950249041541 & 0.029842231365 & -0.152442685886 \\
0.408727382618 & 0.078608384031 & 0.686828518953 \\
0.702869838500 & 0.710682040006 & -0.030083685623 \\

\end{tabular}
\end{coordinates}
\vspace{0.1in}

\begin{coordinates}{ $ \bm{10_{103}} $ }
\begin{tabular}{lll}

0.000000000000 & 0.000000000000 & 0.000000000000 \\
1.000000000010 & 0.000000006065 & 0.000000003547 \\
0.018960290862 & 0.193806840555 & 0.000000000651 \\
0.356936968390 & -0.339698568259 & 0.775334602104 \\
0.540621665830 & -0.116630990376 & -0.182005874232 \\
-0.242304947670 & 0.394584837397 & 0.172513935286 \\
0.500619216492 & -0.272128441912 & 0.232155276040 \\
0.427964164870 & 0.566980004082 & -0.306934938447 \\
-0.083486660694 & -0.039400806722 & 0.301933163297 \\
0.106933714219 & 0.942176862553 & 0.317597135441 \\

\end{tabular}
\end{coordinates}
\vspace{0.1in}

\begin{coordinates}{ $ \bm{10_{105}} $ }
\begin{tabular}{lll}

0.000000000000 & 0.000000000000 & 0.000000000000 \\
0.999999999997 & 0.000000008278 & 0.000000004627 \\
0.421474190912 & 0.815664086349 & 0.000000006316 \\
0.626988867636 & -0.120351838092 & 0.285723486074 \\
-0.057133112140 & 0.219817823101 & -0.359459990455 \\
0.715577625773 & 0.325577919270 & 0.266425716019 \\
0.144478816059 & -0.406796214850 & -0.104349522174 \\
0.988889566369 & 0.100468515784 & -0.276549921967 \\
0.394694417825 & -0.065223071327 & 0.510519594232 \\
0.809926204303 & 0.569219987375 & -0.141450166410 \\

\end{tabular}
\end{coordinates}
\vspace{0.1in}

\begin{coordinates}{ $ \bm{10_{106}} $ }
\begin{tabular}{lll}

0.000000000000 & 0.000000000000 & 0.000000000000 \\
1.000000000010 & -0.000000009907 & 0.000000003092 \\
0.311960823194 & 0.725673533203 & -0.000000006234 \\
0.141613105169 & -0.205341943516 & -0.322787609860 \\
0.114163390974 & 0.396089272152 & 0.475665272891 \\
0.514897953039 & -0.206195163554 & -0.214744220993 \\
0.701365250085 & 0.562190250771 & 0.397476165920 \\
0.197666883183 & 0.592350574934 & -0.465876780466 \\
0.141122674381 & 0.244419753037 & 0.469936713561 \\
0.577640075963 & 0.774859214651 & -0.256758914375 \\

\end{tabular}%
\end{coordinates}
\vspace{0.1in}

\end{multicols*}

\end{document}